\documentclass[a4paper,12pt]{amsart}
\usepackage{graphicx}
\usepackage{amsmath,amssymb,amsthm,amsfonts}
\usepackage{amssymb}
\usepackage[active]{srcltx}
\usepackage{todonotes}
\usepackage[hyperpageref]{backref}
\usepackage[colorlinks=true,urlcolor=blue,citecolor=red,linkcolor=blue,linktocpage,pdfpagelabels,bookmarksnumbered,bookmarksopen]{hyperref}
\usepackage{esint}
\usepackage{latexsym}
\newtheorem{thm}{Theorem}[section]

\newtheorem{lem}[thm]{Lemma}
\newtheorem{prop}[thm]{Proposition}
\newtheorem{rem}[thm]{Remark}
\newtheorem{defn}[thm]{Definition}
\newcommand{\R}{{\mathbb{R}}}
\newcommand{\N}{{\mathbb{N}}}

 \numberwithin{equation}{section}

\begin{document}

\parindent 0pc
\parskip 6pt
\overfullrule=0pt

\title[Classification of singular solutions]{Classification of singular solutions to the subcritical Lane-Emden equation in $\R^N \setminus \{0\}$}

\author{Francesco Esposito$^{*}$, Luigi Muglia$^{*}$, Berardino Sciunzi$^{*}$, Domenico Vuono$^{*}$}
\address{$^{*}$Dipartimento di Matematica e Informatica, UNICAL, Ponte Pietro  Bucci 31B, 87036 Arcavacata di Rende, Cosenza, Italy}
\email{francesco.esposito@unical.it, luigi.muglia@unical.it, berardino.sciunzi@unical.it, domenico.vuono@unical.it}
\thanks{The authors are supported by Gruppo Nazionale per l'Analisi Matematica, la Probabilità. F. Esposito, B. Sciunzi and D. Vuono is partially supported by \emph{INdAM-GNAMPA Esistenza, regolarità e proprietà qualitative per problemi non lineari} E5324001950001. L. Muglia is partially supported by \emph{INdAM-GNAMPA Esistenza e proprietà qualitative per soluzioni di equazioni ellittiche con termini di ordine inferiore} E5324001950001. F. Esposito is partially supported by the PID Project (Spain) PID2024-155314NB-I00.
}




\maketitle

\date{\today}

\begin{abstract}
We provide a  classification result of positive solutions with a non-removable singularity at the origin to the subcritical Lane-Emden equation $-\Delta_p u = u^q$ in $\mathbb{R}^N \setminus \{0\}$, where $p\in(1,N)$, $q\in (p-1,p^*-1)$ and $u^{q-p+1} \in L^{\frac{N}{p}}(\mathbb{R}^N \setminus B_R)$.
\end{abstract}

\section{Introduction and main results}

In recent years, the study of non-negative solutions to the classical Lane-Emden equation
\begin{equation}\label{eq:pLane-Emden}
-\Delta_p u = u^q \quad \text{in } \mathbb{R}^N, 
\end{equation}
has attracted significant attention, where $p \in (1,N)$ and $q \in [p-1,p^*-1]$, with $p^*= pN/(N-p)$ being the Sobolev critical exponent. When $q=p^*-1$, it is well-known that the \textit{Aubin-Talenti bubbles}
%
%
form a family of solutions to \eqref{eq:pLane-Emden}, which realize the equality in the sharp Sobolev inequality in $\mathbb{R}^N$.

Let us start our analysis with the semilinear case, i.e.~$p=2$. In~\cite{GNN81}, the authors proved that any non-negative solution $u$ to \eqref{eq:pLane-Emden}, such that $u(x) = O\left(1/|x|^{m}\right)$ as $|x| \rightarrow +\infty$, for some $m >0$, is radially symmetric and decreasing about some point in $\mathbb{R}^N$. The proof is based on a refinement of the celebrated moving plane procedure developed by the authors in a previous paper~\cite{GNN}. Later on, when $q=2^*-1$, Caffarelli, Gidas, and Spruck in~\cite{CGS} classified all the positive solutions to \eqref{eq:pLane-Emden}. The breakthrough idea was the use of the Kelvin transform, which allowed them to show that the moving plane procedure can start. As a consequence, they proved that every positive solution $u \in H^1_{loc}(\mathbb{R}^N)$ to \eqref{eq:pLane-Emden} in the critical case is an Aubin-Talenti bubble, implying that solutions are unique up to translations and scaling.

The situation in the subcritical regime, i.e.~when $q \in [1, 2^*-1)$, is completely different. In the seminal paper~\cite{GS}, Gidas and Spruck proved that any non-negative solution of \eqref{eq:pLane-Emden} in $\mathbb{R}^N$ must be trivial. The proof of this result is based on the use of integral identities, a technique that goes back to the work of Obata~\cite{Obata}.

Let us focus our attention on the quasilinear case, i.e.~$p \neq 2$. The critical $p$-Laplace equation, i.e.~when $q = p^*-1$, has been the object of several studies in the differential geometry and PDE communities; indeed, problem \eqref{eq:pLane-Emden} is related to the study of the critical points of the Sobolev inequality. As in the semilinear framework, we remark that an interesting and challenging problem is the classification of solutions to \eqref{eq:pLane-Emden}. As before, the $p$-Aubin-Talenti bubbles form a two-parameter family of solutions to \eqref{eq:pLane-Emden}. 

The classification of solutions in the critical regime is much more complicated, since the use of the Kelvin transformation is not allowed. Under the finite energy assumption, i.e.~$u \in \mathcal{D}^{1,p}(\mathbb{R}^N)$, the classification of all positive solutions to \eqref{eq:pLane-Emden} has been an open and challenging problem, which was recently solved in a series of papers by Damascelli, Merch\'an, Montoro, and Sciunzi~\cite{DMMS} for $p \in (2N/(N+2), 2)$, V\'etois~\cite{Vet} for $p \in (1, 2)$, and Sciunzi~\cite{Sciu16} for $p \in (2, N)$. In particular, in these articles it was proved that any positive solution to \eqref{eq:pLane-Emden} with $q = p^*-1$ that belongs to $\mathcal{D}^{1,p}(\mathbb{R}^N)$ is a $p$-Aubin-Talenti bubble. 
More recently, Ciraolo, Figalli, and Roncoroni~\cite{CFR} used a strategy based on integral estimates to extend the classification of positive $\mathcal{D}^{1,p}$-solutions to a class of anisotropic $p$-Laplace-type equations in convex cones. A similar approach was used in~\cite{Cat} to obtain new classification results for positive, weak solutions to \eqref{eq:pLane-Emden} which are not a priori in $\mathcal{D}^{1,p}(\mathbb{R}^N)$. In particular, in~\cite{Cat}, Catino, Monticelli and Roncoroni managed to obtain the complete classification of positive, weak solutions to \eqref{eq:pLane-Emden} in the case where $N = 2$, or $N = 3$ and $p \in (3/2, 2)$. This method was recently improved by Ou ~\cite{Ou}, who managed to extend the result to the case where $p \geq (N+1)/3$, and by V\'etois~\cite{Vet2} for $N \geq 4$ and $p \geq p_N$ for some number $p_N \in (N/3, (N+1)/3)$ such that $p_N \sim N/3 + 1/N$. We also refer the reader to the recent interesting results in \cite{CatMonRon2,CG}.

The situation in the subcritical regime, i.e.~when $q \in [p-1, p^*-1)$, is analogous to the one analyzed in the semilinear framework; however, the nonlinear nature of the problem introduces several difficulties. In the seminal paper~\cite{SZ}, Serrin and Zou proved the natural counterpart to the result of Gidas and Spruck for the quasilinear case. More precisely, they showed that any non-negative solution to \eqref{eq:pLane-Emden} in $\mathbb{R}^N$ with $p \in (1,N)$ and $q \in [p-1,p^*-1)$ must be identically zero.

The aim of the present paper is to provide a complete classification of positive solutions with a non-removable singularity to the punctured subcritical equation in $\mathbb{R}^N$. More precisely, we study non-negative solutions to
\begin{equation}\label{eq:puncturedLaneEmden}
    -\Delta_p u = u^q \quad \text{in } \mathbb{R}^N \setminus \{0\},
\end{equation}
where $p \in (1,N)$ and $q \in [p-1,p^*-1)$. In particular, our main result is:
\begin{thm}\label{Teo:simmetriamenoilpunto}
Let $p \in (1,N)$ and $q \in \left(N(p-1)/(N-p), p^*-1\right)$. Let $u \in W^{1,p}_{\text{loc}}(\mathbb{R}^N \setminus \{0\})$ be a positive weak solution to \eqref{eq:puncturedLaneEmden} with a non-removable singularity at the origin, such that $u^{q-p+1} \in L^{\frac{N}{p}}(\mathbb{R}^N \setminus B_R)$ . Then 
\[
u=u(r)\qquad with \quad r=|x|, \quad u'(r)<0
\]
and, up to scaling, $u$ is given by the solution with
\begin{equation}
\lim_{r \to 0^+} r^{\frac{p}{q-p+1}} u(r) = 1 , \qquad \lim_{r\to +\infty} r^{\frac{N-p}{p-1}} u(r)  = \beta > 0.
\end{equation}
\end{thm}
The proof of the radial symmetry of the solutions is based on a refined adaptation of the moving plane procedure. This technique goes back to the seminal works of Alexandrov and Serrin~\cite{A,serrinMov}, as well as the celebrated contributions by Berestycki and Nirenberg~\cite{BN}, and Gidas, Ni, and Nirenberg~\cite{GNN}. To exploit this technique, we first need to establish some a priori estimates for both the solution $u$ and its gradient; then, we employ a truncation argument developed in~\cite{EMS}. Once the radial symmetry is established, the classification result follows by applying a result due to M.F.~Bidaut-V\'eron~\cite[Theorem~5.2]{BV}.

Our result is, in a sense, sharp, since Serrin and Zou in~\cite{SZ} proved that equation \eqref{eq:puncturedLaneEmden} admits a non-trivial solution if and only if $q > N(p-1)/(N-p)$. Thus in the range $q \in [p-1,N(p-1)/(N-p)]$ there are no positive singular solutions. This implies that below this threshold, positive singular solutions do not exist either. For the non-existence results we refer also to \cite{BV, BVP}.

Theorem \ref{Teo:simmetriamenoilpunto} will be proved as a consequence of more general results. To this end, we will study the more general problem
\begin{equation}\label{eq:Prob} \tag{$\mathcal{P}_\Gamma$}
	\begin{cases}
		\displaystyle -\Delta_p u = c(x) u^{q} \qquad &\text{in } \mathbb{R}^N \setminus \Gamma,\\
		u > 0 & \text{in } \mathbb{R}^N \setminus \Gamma,
	\end{cases}
\end{equation}
where $p \in (1,N)$ and $q \in \left(N(p-1)/(N-p), p^*-1\right)$. The set $\Gamma$ is a compact subset of $\mathbb{R}^N$ satisfying suitable assumptions, referred to as the {\em critical set}. The function $c(\cdot) \in L^\infty(\mathbb{R}^N)$ is non-negative.
Moreover, the solution $u$ has a  singularity on the critical set $\Gamma$. Throughout the paper, we assume that
\begin{equation}\label{ipotesiu} \tag{$\mathcal{H}_u$}
    u^{q-p+1} \in L^{\frac{N}{p}}(\mathbb{R}^N \setminus B_R)
\end{equation}
for every ball $B_R$ of radius $R > R_1$ centered at the origin, where $R_1 > 0$ denotes a radius such that the open ball $B_{R_1}$ contains $\Gamma$. 

We remark that solutions to \eqref{eq:Prob} are to be understood in the weak distributional sense. To be more precise, we give the following definition.
\begin{defn}\label{Definizione_debole}
We say that $u \in W^{1,p}_{loc}(\R^N \setminus \Gamma)$ is a weak solution of \eqref{eq:Prob} if it satisfies the following
 \begin{equation}\label{wsol}
\int_{\R^N}|\nabla u|^{p-2}\langle \nabla u,\nabla \varphi\rangle dx=\int_{\R^N}c(x) u^q\varphi dx, \qquad \forall \varphi\in C_c^\infty(\R^N\setminus \Gamma ).
\end{equation} 
\end{defn}

\begin{rem}\label{primorem} Let us note that, by Serrin's result \cite[Theorem~1]{serrin}, any weak 
solution of \eqref{eq:Prob} satisfies 
\(u \in L^{\infty}_{\mathrm{loc}}(\mathbb{R}^N \setminus \Gamma)\).  
Then, by classical regularity theory (see \cite{DB, Tolk}), we have that 
\(u \in C^{1,\alpha}_{\mathrm{loc}}(\mathbb{R}^N \setminus \Gamma)\).
\end{rem}

We stress that the standard moving plane technique cannot be applied straightforwardly, mainly for three reasons. First of all, the arguments used in~\cite{EFS, SciunziJMPA}, which work for the case $p=2$, strongly rely on a homogeneity argument that fails when $p \neq 2$. Furthermore, since the gradient of the solution may blow up near the critical set, the equation can exhibit both a degenerate and a singular nature at the same time; hence, we approach the problem inspired by the case of bounded domains treated in~\cite{EMS}. Finally, since we work in the entire space $\mathbb{R}^N$, we need to establish some a priori estimates for the solutions to \eqref{eq:Prob} and their gradients.

Hence, we start by stating the following result.

\begin{thm} \label{thm:SolGradEst}
Let $c(\cdot)$ be a nonnegative bounded function, and let $u$ be a weak solution of \eqref{eq:Prob}, satisfying the assumption \eqref{ipotesiu}. Then there exists a positive constant $\hat R>R_1$ depending on $N, p,\|c\|_{\infty},q$ and  $u$, such that
\begin{equation}\label{eq:estSol}
	\frac{\hat c}{|x|^{\frac{N-p}{p-1}}}  \leq u(x) \leq \frac{\hat C}{|x|^{\frac{N-p}{p-1}}} \qquad x \in B_{\hat R}^c,
\end{equation}
\begin{equation}\label{eq:estatInfgradthm}
	\frac{\tilde c}{|x|^{\frac{N-p}{p-1}+1}}  \leq |\nabla u(x)| \leq \frac{\tilde C}{|x|^{\frac{N-p}{p-1}+1}} \qquad x \in B_{\hat R}^c,
\end{equation}
where $\hat c$, $\hat C$,$\tilde c$, $\tilde C$ are positive constants depending on $N,p, q, \|c\|_{\infty}$ and $u$.
\end{thm}
Having established the a priori estimates, we can now state the symmetry result for the case $p > 2$. Specifically, we prove the following:
\begin{thm}\label{Teopmaggiore2}
Let $p \geq 2$  and let $u$ be a weak solution to \eqref{eq:Prob}, satisfying \eqref{ipotesiu}. Assume that $c(x)$ is a positive bounded function, even w.r.t. the $x_1$-direction and non-decreasing in the set $\{x_1<0\}$. If $\Gamma \subset \{x_1 = 0\}$ is closed and such that $\mathrm{Cap}_p(\Gamma) = 0$, then it follows that $u$ is symmetric with respect to the hyperplane $\{x_1 = 0\}$ and increasing in the $x_1$-direction in $\R^N \cap \{x_1 < 0\}$.
\end{thm}

Surprisingly, the case $p \in (1,2)$ presents more difficulties, due to the fact that, as already remarked, the operator may degenerate near the critical set even if $p \in (1,2)$. Inspired by~\cite{EMS}, we state the following result:
\begin{thm}\label{Teopminore2}
Let $p<2$ and let $u$ be a weak solution to \eqref{eq:Prob}, satisfying \eqref{ipotesiu}. Assume that $c(x)$ is a positive bounded function, even w.r.t. the $x_1$-direction and non-decreasing in the set $\{x_1<0\}$. Moreover, let us assume that $\Gamma\subset\{x_1=0\}$ is closed and that $\Gamma\subseteq\mathcal{M}$ for some compact 
$C^2$ submanifold $\mathcal{M}$ of dimension $m\le N-k$, with $k\ge N/2$ for $N>2$ and $\Gamma=\{0\}$ for $N=2$ (or finite). 
Then $u$ is symmetric with respect to the hyperplane $\{x_1=0\}$ and increasing in the  $x_1$–direction in $\mathbb{R}^N\cap\{x_1<0\}$.
\end{thm}

The paper is structured as follows:

\begin{itemize}
	\item In Section~\ref{preliminaryEst}, we recall some technical lemmas and prove several preliminary estimates that will be crucial for the proof of Theorem~\ref{thm:SolGradEst}.
	
	\item In Section~\ref{AsympEst}, we establish the a priori estimates at infinity for solutions to \eqref{eq:Prob} and their gradients, thereby proving Theorem~\ref{thm:SolGradEst}.
	
	\item In Section~\ref{sec:symmetry}, we prove the symmetry results stated in Theorems~\ref{Teo:simmetriamenoilpunto}, \ref{Teopmaggiore2}, and~\ref{Teopminore2}.
\end{itemize}

\section{Asymptotic estimates of solutions} \label{preliminaryEst}

The aim of this section is to prove some preliminary estimates at infinity that will be crucial in the proofs of the main results. To do this, it is not restrictive to consider  the following problem in an exterior domain
\begin{equation}\label{eq:Probexterior} \tag{$\mathcal{P}_{R_1}$}
	\begin{cases}
		\displaystyle -\Delta_p u= c(x) u^{q} \quad &\text{in } \R^N\setminus B_{R_1}\\
		u>0 & \text{in } \R^N\setminus B_{R_1}.
	\end{cases}
\end{equation}
A solution to \eqref{eq:Probexterior} is understood in the sense of
Definition~\ref{Definizione_debole}; namely, we assume that  
\(c(x)\in L^{\infty}(\mathbb{R}^N\setminus B_{R_1})\) is nonnegative and that  
\(u\in W^{1,p}_{\mathrm{loc}}(\mathbb{R}^N\setminus B_{R_1})\) satisfies the
assumption \eqref{ipotesiu}.  
Under these conditions, \(u\) is a weak solution if
\begin{equation}\label{wsolparte2}
\int_{\mathbb{R}^N} |\nabla u|^{p-2}
\langle \nabla u, \nabla \varphi\rangle\,dx
=
\int_{\mathbb{R}^N} c(x)\,u^{q}\,\varphi\,dx,
\qquad
\forall\,\varphi\in C_c^\infty(\mathbb{R}^N\setminus B_{R_1}).
\end{equation}
We remark that the radius \(R_1\) can be taken sufficiently large in the
applications, since the solutions of problem \eqref{eq:Prob} will be compared with
solutions of \eqref{eq:Probexterior}; in particular, the singular set
\(\Gamma\) does not play any role in this exterior analysis.

Now we recall a lemma that will be very useful in the proofs of our results; it is a well-known result and we refer the reader for instance to \cite[Lemma 4.19]{HanLin}. 

\begin{lem}\label{lem:nondec}
Let $\mathcal{L}$ and $g$ be two nondecreasing functions on the interval $(0, \bar R]$, for some $\bar R>0$. Suppose that it holds
$$\mathcal{L}(\tau R) \leq \sigma \mathcal{L}(R) + g(R) \quad \text{for all } R \leq \bar R,$$
for some $0 < \sigma, \tau < 1$. Then, for any $\mu \in (0,1)$ and $R \leq \bar R$ we have
$$\mathcal{L}(R) \leq \frac{1}{\sigma} \left(\frac{R}{\bar R}\right)^\alpha \mathcal{L}(\bar R) + \frac{1}{1-\sigma}g(\bar R^\mu R^{1-\mu})$$
where $\alpha = \alpha (\sigma, \tau, \mu) =(1-\mu) \log \sigma / \log \tau$.
\end{lem}

We begin with the following lemma

\begin{lem}\label{lem:aux1}
	 There exists a positive constant $\tau$ depending only on $N, p, \|c\|_{\infty},q$ such that for any $\bar{R}> R_1$ and for any solution $u$  to problem \eqref{eq:Probexterior} satisfying 
	\begin{equation}\label{eq:assumptionsmall}
		\|u^{q-p+1}\|_{L^{\frac Np} (\R^N \setminus {B}_{\bar R})} \leq \tau , 
	\end{equation} 
	 there exists a positive constant $\mathcal{C}$ depending only on $N, p, \|c\|_{\infty},q$ and $\overline R$ such that 
\begin{equation}  \label{eq:EstAtInfty}
       	\|u^{q-p+1}\|_{L^{\frac{N}{p}}({B}_R^c)} \leq \frac{\mathcal{C}}{R^{\sigma_2}} \quad \text{for } R \geq \bar R, 
	\end{equation} 
	where $\sigma_2$ is a positive constant depending on $N, p, q$.
\end{lem}

\begin{proof}
	Let us fix two radii $R$ and $\tilde R$ such that $R/2>R_1$ and $\tilde R>R$. We introduce two cut–off functions as follows. First, let \(\eta\in C^\infty(\mathbb{R}^N)\) be such that
    \begin{equation}\label{eq:cutoff}
    	\left\{\begin{array}{llll}
		0 \leq \eta \leq 1 \quad &\text{in } \R^N\\
		\eta \equiv 1 \quad &\text{in } {B}_R^c\\
		\eta \equiv 0 \quad &\text{in } {B}_{R / 2}\\
		|\nabla \eta| \leq \frac{2}{R} &\text{in } {B}_R \setminus {B}_{R / 2}.
	\end{array}\right.
    \end{equation}
 Next, since the test function that we shall introduce later is not compactly
supported unless we truncate it at large scales, we also consider a second
cut–off \(\psi_{\tilde R}\in C^\infty_c(\mathbb{R}^N)\) defined by
\begin{equation}\label{eq:cutoffR}
    	\left\{\begin{array}{llll}
		0 \leq \psi_{\tilde R} \leq 1 \quad &\text{in } \R^N\\
		\psi_{\tilde R} \equiv 1 \quad &\text{in } {B}_{{\tilde R}}\\
		|\nabla \psi_{\tilde R}| \leq \frac{2}{{\tilde R}} &\text{in } {B}_{2{\tilde R}} \setminus {B}_{{\tilde R}}\\
		\psi_R \equiv 0 \quad &\text{in } {B}_{2{\tilde R}}^c.
	\end{array}\right.
    \end{equation}
    The introduction of \(\psi_{\tilde R}\) is necessary to ensure that the test
function we shall use is admissible in the weak formulation.  
After performing all the required estimates, we will eventually let 
\(\tilde R\to +\infty\).
	With this in mind, we now consider $\beta>0$ and 
	$$\varphi:=\eta^p u^\beta\psi_{\tilde R}^p.
	$$ 
We note that $\varphi$ is a good test function; this follows by Remark \ref{primorem} and by the fact that $u>0$. So, we have 
    $$\nabla \varphi=\beta\eta^p u^{\beta-1}\psi_{\tilde R}^p\nabla u+p\eta^{p-1} u^{\beta}\psi_{\tilde R}^p\nabla \eta+p\eta^p u^{\beta-1}\psi_{\tilde R}^{p-1}\nabla \psi_{\tilde R}.$$

    Using $\varphi$ as test function in \eqref{wsolparte2}, we obtain
	\begin{equation}\label{eq:weakForTest1}
    \begin{split}
		&\beta\int_{\R^N}\eta^p u^{\beta-1}\psi_{\tilde R}^p|\nabla u|^{p}  \, dx=-\int_{\R^N} p \psi_{\tilde R}^p\eta^{p-1} u^\beta |\nabla u|^{p-2} \langle \nabla u ,  \nabla \eta  \rangle  \, dx
        \\&+\int_{\R^N} p \psi_{\tilde R}^{p-1}\eta^{p} u^\beta |\nabla u|^{p-2} \langle \nabla u ,  \nabla \psi_{\tilde R}  \rangle  \, dx+ \int_{\R^N} c(x)u^{q+\beta} \psi_{\tilde R}^p\eta^p \, dx.
        \end{split}
	\end{equation}

Using Cauchy-Schwarz's inequality, we get
	\begin{equation}\label{eq:weakForTest2}
		\begin{split}
			\beta\int_{\R^N}\eta^p u^{\beta-1}\psi_{\tilde R}^p|\nabla u|^{p}  \, dx&\leq \int_{\R^N} p\psi_{\tilde R}^p \eta^{p-1} u^{\beta} |\nabla u|^{p-1} |\nabla \eta| \, dx 
            \\&+\int_{\R^N} p\psi_{\tilde R}^{p-1} \eta^{p} u^{\beta} |\nabla u|^{p-1} |\nabla \psi_{\tilde R}| \, dx+ \int_{\R^N} c(x)u^{q+\beta} \psi_{\tilde R}^p\eta^p\, dx.
		\end{split}
	\end{equation} 
By the weighted Young's inequality, that is 
\(ab \leq \varepsilon\, a^{\frac{p}{p-1}} + \mathcal{C}_\varepsilon\, b^{p}\) 
for \(a,b \geq 0\), with \(\mathcal{C}_\varepsilon>0\), applied to the first and second terms on the right-hand side of \eqref{eq:weakForTest2}, and rewriting $\beta=(\beta-1)(p-1)/p+(\beta+p-1)/p$ 
we obtain, for any \(0<\varepsilon<1\),
\begin{equation}\label{eq:weakForTest33}
	\begin{split}
		\beta\int_{\R^N}\eta^p u^{\beta-1}\psi_{\tilde R}^p|\nabla u|^{p}  \, dx  & \le 2p\varepsilon   \int_{\R^N}\eta^p u^{\beta-1}\psi_{\tilde R}^p|\nabla u|^{p}\,dx + C(p,\varepsilon) \int_{\R^N}  \psi_{\tilde R}^p|\nabla \eta|^p u^{\beta+p-1} \, dx \\
		&+C(p,\varepsilon) \int_{\R^N}  \eta^p|\nabla \psi_{\tilde R}|^p u^{\beta+p-1} \, dx + \int_{\R^N} 	c(x)u^{q+\beta} \psi_{\tilde R}^p\eta^p \, dx,
	\end{split}
\end{equation} 
where $ C(p,\varepsilon):=p\mathcal{C}_\varepsilon$. Hence, taking $\varepsilon$ small enough,  we deduce 
\begin{equation}\label{eq:weakForTest3}
	\begin{split}
(\beta-2p\varepsilon)\int_{\R^N}\eta^p u^{\beta-1}\psi_{\tilde R}^p|\nabla u|^{p}  \, dx  & \le  C(p,\varepsilon) \int_{\R^N}  \psi_{\tilde R}^p|\nabla \eta|^p u^{\beta+p-1} \, dx \\
		&+C(p,\varepsilon) \int_{\R^N}  \eta^p|\nabla \psi_{\tilde R}|^p u^{\beta+p-1} \, dx +\int_{\R^N} 	c(x)u^{q+\beta} \psi_{\tilde R}^p\eta^p \, dx.
	\end{split}
\end{equation} 
Now, we note that for $p>1$ there exists a positive constant $C(p,\beta)$, such that 
$$\left|\nabla\left(\eta\psi_{\tilde R}u^{\frac{\beta-1}{p}+1}\right)\right|^p\leq C(p,\beta)\left(\eta^p\psi_{\tilde R}^pu^{\beta-1}|\nabla u|^p+\psi_{\tilde R}^pu^{\beta+p-1}|\nabla \eta|^p+\eta^pu^{\beta+p-1}|\nabla \psi_{\tilde R}|^p\right).$$
Thanks to last inequality \eqref{eq:weakForTest3} becomes
\begin{equation}\label{eq:lemmaTypeYoung}
\begin{split}
	\int_{\R^N} \left|\nabla\left(\eta\psi_{\tilde R}u^{\frac{\beta-1}{p}+1}\right)\right|^p \,dx &\leq \frac{C(p,\beta)\|c\|_{\infty}}{\beta-2p\varepsilon} \int_{\R^N} u^{q-p+1}\psi_{\tilde R}^pu^{\beta+p-1} \eta^p\,dx
    \\&+C(p,\beta)\left(1+\frac{C(p,\varepsilon)}{(\beta-2p\varepsilon)}\right) \int_{\R^N} \psi_{\tilde R}^p |\nabla \eta|^p u^{\beta+p-1} \, dx
    \\&+C(p,\beta)\left(1+\frac{C(p,\varepsilon)}{(\beta-2p\varepsilon)}\right)\int_{\R^N}  \eta^p|\nabla \psi_{\tilde R}|^p u^{\beta+p-1} \, dx.
    \end{split}
\end{equation}

Now, using the Sobolev inequality in the left hand side of \eqref{eq:lemmaTypeYoung} and by the H\"older's inequality with exponents $(p^*/p,N/p)$ in the right hand side, we obtain
\begin{equation}\label{eq:weakForTest8}
	\begin{split}
	&\frac{1}{\mathcal{C}^p_S}\left(\int_{\R^N} \left(\eta\psi_{\tilde R}u^{\frac{\beta+p-1}{p}}\right)^{p^*} \,dx\right)^\frac{p}{p^*} 
    \leq \int_{\R^N}\left|\nabla\left(\eta\psi_{\tilde R}u^{\frac{\beta-1}{p}+1}\right)\right|^p \,dx
    \\&\leq \frac{C(p,\beta)\|c\|_{\infty}}{\beta-2p\varepsilon} \left[\left(\int_{B_{R/2}^c} u^{\frac{(q-p+1)N}{p}}dx\right)^{\frac{p}{N}}\left(\int_{\R^N} \left(\eta\psi_{\tilde R}u^{\frac{\beta+p-1}{p}}\right)^{p^*}dx\right)^{\frac{p}{p^*}} \right]\\
&\quad  +C(p,\beta)\left(1+\frac{C(p,\varepsilon)}{(\beta-2p\varepsilon)}\right) \left[\left(\int_{B_R\setminus B_{R/2}}  \left(\psi_{\tilde R}u^{\frac{\beta+p-1}{p}}\right)^{p^*} \, dx\right)^{\frac{p}{p^*}} \left(\int_{\R^N}  |\nabla \eta|^{N} \, dx\right)^{\frac{p}{N}}\right]
\\&\quad +C(p,\beta)\left(1+\frac{C(p,\varepsilon)}{(\beta-2p\varepsilon)}\right) \left[\left(\int_{B_{2\tilde R}\setminus B_{\tilde R}}  \left(\eta u^{\frac{\beta+p-1}{p}}\right)^{p^*} \, dx\right)^{\frac{p}{p^*}} \left(\int_{\R^N}  |\nabla \psi_{\tilde R}|^{N} \, dx\right)^{\frac{p}{N}}\right],
\end{split}
\end{equation}
where $\mathcal{C}_S$ denotes the Sobolev constant. Now we choose $\beta>0$, such that 
$$(\beta+p-1)p^*/p=(q-p+1)N/p.
$$
We remark that, if $q>N(p-1)/(N-p)$, then $\beta >0$. With this choice, under the assumption \eqref{ipotesiu}, we have
\begin{equation}\label{cosepazze}
    \int_{B_{2\tilde R}\setminus B_{\tilde R}}  \left(\eta^{p^*} u^{\frac{(q-p+1)N}{p}}\right) \, dx\rightarrow 0,
\end{equation}
for $\tilde R\rightarrow +\infty$. By \eqref{eq:weakForTest8} and \eqref{cosepazze}, and by dominated convergence Theorem for $\tilde R\rightarrow +\infty$, we deduce that
\begin{equation}\label{eq:weakForTest9}
	\begin{split}
	&\left(\int_{\R^N} \eta^{p^*} u^{\frac{(q-p+1)N}{p}}\,dx\right)^\frac{p}{p^*} 
   \\&\leq \frac{C(p,\beta)\|c\|_{\infty}\mathcal{C}^p_S}{\beta-2p\varepsilon} \|u^{q-p+1}\|_{L^{\frac NP}(B_{R/2}^c)}\left(\int_{\R^N} \eta^{p*} u^{\frac{(q-p+1)N}{p}}dx\right)^{\frac{p}{p^*}} \\
& \quad+C(p,\beta)\mathcal{C}^p_S\left(1+\frac{C(p,\varepsilon)}{(\beta-2p\varepsilon)}\right) \left[\left(\int_{{B}_R \setminus {B}_{R / 2}}  u^{\frac{(q-p+1)N}{p}} \, dx\right)^{\frac{p}{p^*}} \left(\int_{\R^N}  |\nabla \eta|^{N} \, dx\right)^{\frac{p}{N}}\right]
\\&\leq \frac{C(p,\beta)\|c\|_{\infty}\mathcal{C}^p_S}{\beta-2p\varepsilon} \|u^{(q-p+1)}\|_{L^{\frac NP}(B_{R/2}^c)}\left(\int_{\R^N} \eta^{p*} u^{\frac{(q-p+1)N}{p}}dx\right)^{\frac{p}{p^*}} \\
&\quad +C(p,\beta)\mathcal{C}^p_SC(p,N)\left(1+\frac{C(p,\varepsilon)}{(\beta-2p\varepsilon)}\right) \left(\int_{{B}_R \setminus {B}_{R / 2}} u^{\frac{(q-p+1)N}{p}} \, dx\right)^{\frac{p}{p^*}}, 
\end{split}
\end{equation}
where $C(p,N)$ is a positive constant depending on $p$ and $N$.
If $\|c\|_{\infty}\neq 0$, choosing $\overline{R}>0$ sufficiently large such that \eqref{eq:assumptionsmall}
holds, whenever
$$
 \tau:=\frac{\beta-2p\varepsilon}{2C(p,\beta))\|c\|_{\infty}\mathcal{C}^p_S},
$$
then $$\frac{C(p,\beta)\|c\|_{\infty}\mathcal{C}^p_S}{\beta-2p\varepsilon}\|u^{q-p+1}\|_{L^{\frac Np}(B_{R/2}^c)} \leq \frac12$$  
for all $R \geq 2\overline R$. Hence we obtain that
$$\int_{{B}_{R}^c} u^{(q-p+1)\frac Np} \,dx  \leq \int_{\R^N} \eta ^{p^*}u^{(q-p+1)\frac Np} \,dx \leq \bar C \int_{{B}_R \setminus {B}_{R / 2}}  	u^{(q-p+1)\frac Np}  \, dx, \qquad \forall \  R \geq \bar R,$$
where $\bar C := \left(2C(p,\beta)\mathcal{C}^p_SC(p,N)\left(1+\frac{C(p,\varepsilon)}{(\beta-2p\varepsilon)}\right)\right)^{p^*/p}$ depends only on $N$, $p$, and $q$. The previous inequality can be rewritten as
$$
\int_{{B}_{R}^c} u^{(q-p+1)\frac Np} \,dx  \leq \bar C \left[\int_{{B}_{\frac R2}^c} u^{(q-p+1)\frac Np} \,dx  -\int_{{B}_{R}^c} u^{(q-p+1)\frac Np} \,dx\right].
$$
Therefore, we deduce 
$$
\int_{{B}_{R}^c} u^{(q-p+1)\frac Np} \,dx  \leq \theta\int_{{B}_{\frac R2}^c} u^{(q-p+1)\frac Np} \,dx, 
$$
with $\theta:=\frac{\bar{C}}{(1+\bar{C})}\in (0,1)$.
Denoting  $\displaystyle \mathcal{L}(s):=\int_{{B}_{\frac {1}{2s}}^c}u^{(q-p+1)\frac Np} \, dx$,  where $s:=1/R$, we get that
$$\mathcal{L}(s/2) \leq \vartheta \mathcal{L}(s) \qquad \forall \ s \leq \bar s,$$
where $\bar s=1/\bar R$. Moreover, we note that the function \(\mathcal{L}(\cdot)\) is nondecreasing.
 Now, by Lemma \ref{lem:nondec} it follows that
$$\mathcal{L}(s) \leq \frac{1}{\vartheta} \mathcal{L}(\bar s) \left(\frac{s}{\bar s}\right)^{\sigma_2 '} \qquad \forall \ 0 < s \leq \bar s,$$
where $\sigma_2'=\frac{1}{2}\log( \vartheta)/ \log (1/2)$ depends only on $\vartheta$ (we have chosen \(\mu = 1/2\) in Lemma~\ref{lem:nondec}).
Now the thesis follows by setting $\sigma_2 = \sigma_2'p/N$ and $\mathcal{C}=(\vartheta^{-1}\overline s^{-\sigma_2'}\mathcal{L}(\bar s))^{p/N}$. We conclude noting that, if $c(x)= 0$ a.e. in $\R^N$, then the same conclusion holds true for any $\tau>0$.
\end{proof}

In what follows, using an approach similar to \cite[Theorem E.020]{Peral}, we gain some additional integrability on \(u\); see inequality \eqref{eq:EstAt0aux2} below. 
In doing so, we are interested in two aspects: first, obtaining an \(L^\infty\) estimate for \(u\) by means of the approach developed by Serrin \cite{serrin} (see also \cite{moser}); 
second, controlling the constant appearing in inequality \eqref{eq:EstAt0aux2}. 
This will allow us to derive an initial decay at infinity for the solution \(u\) of problem \eqref{eq:Probexterior}; see Theorem \ref{teoo}.

Now, we denote by $\mathcal{A}_R={B}_{8R} \setminus {B}_{R/8}$, $\mathcal{D}_R={B}_{4R} \setminus {B}_{R/4}$ for $R>0$ and
$$
\beta:=\frac{(q-p+1)N}{p}.
$$
\begin{lem} \label{aux2}
	Let $t >\beta$, then there exists a positive constant $\sigma = \sigma(N,p,t,\|c\|_{\infty})$ such that for any solution $u$  to problem \eqref{eq:Probexterior} and for any $\bar R> R_1$ such that 
	\begin{equation}\label{eq:assumptionsmallaux2}
		\|u^{q-p+1}\|_{L^{\frac Np} (\R^N \setminus {B}_{\bar R})} \leq \sigma,
	\end{equation} 
    then 
	\begin{equation} \label{eq:EstAt0aux2}
		\left(\fint_{\mathcal{D}_R} u^t \, dx \right)^{\frac{1}{t}} \leq \mathcal{C} \left(\fint_{\mathcal{A}_R} u^{(q-p+1)\frac{N}{p}} \, dx \right)^{\frac{p}{(q-p+1)N}}, \qquad \forall  R>8 \bar R,
	\end{equation} 
	where $\displaystyle \fint_{B} u^t \, dx = \frac{1}{|B|} \int_{B} u^t \, dx$ and $\mathcal{C}$ is a positive constant  depending only on $N, p, \bar R,t$ and $\|c\|_{\infty}$ .
\end{lem}

\begin{proof}
	To begin the proof, we note that, setting $\hat u(x)=u(Rx)$ and $\hat c(x)=c(Rx)$, for $R>8 \bar R$, we deduce 
	\begin{equation}\label{eq:weakforspecial}
		 \begin{split}
			\int_{\mathcal{A}_1} |\nabla \hat u|^{p-2}\langle \nabla \hat u ,  \nabla \varphi  \rangle  \, dx 
			= R^p\int_{\mathcal{A}_1} 	\hat c(x)\hat u^{q} \varphi \, dx \qquad \forall \varphi \in C^\infty_c(\mathcal{A}_1).
		\end{split}
	\end{equation}
	 Since $u$ is positive, by a density argument, for any \(\eta \in C^\infty_c(\mathcal{A}_1)\) it is possible to choose 
$$
\varphi = \eta^p\, \hat u^{\,p(s-1)+1},
$$
with \(s > (p-1)/p\), as a test function in 
\eqref{eq:weakforspecial}, see Remark \ref{primorem}. Hence, we compute the left-hand-side of \eqref{eq:weakforspecial}
	\begin{equation}\label{eq:MoserI_1}
		\begin{split}
			\int_{\mathcal{A}_1} |\nabla \hat u|^{p-2}\langle \nabla \hat u ,  \nabla \varphi  \rangle  \, dx=&
			p\int_{\mathcal{A}_1} |\nabla \hat u|^{p-2}\langle \nabla \hat u , \nabla \eta  \rangle \eta^{p-1} \hat u^{p(s-1)+1}  \, dx
			\\&+ [p(s-1)+1] \int_{\mathcal{A}_1} |\nabla \hat u|^p  \eta^p \hat u^{p(s-1)}   \, dx.\\
		\end{split}
	\end{equation}

By \eqref{eq:MoserI_1} and using Cauchy-Schwarz's inequality, we obtain
\begin{equation}\label{eq:MoserStep2}
	\begin{split}
		&[p(s-1)+1] \int_{\mathcal{A}_1} |\nabla \hat u|^p  \eta^p \hat u^{p(s-1)}   \, dx \\
		&\leq  p   \int_{\mathcal{A}_1} |\nabla \hat u|^{p-1} \eta^{p-1} \hat u^{(p-1)(s-1)} \hat u^s |\nabla \eta|  \, dx + R^p\|c\|_{\infty}\int_{\mathcal{A}_1} 	\hat u^{q+p(s-1)+1}  \eta^p \, dx.
	\end{split}
\end{equation}
We can now apply a weighted Young's inequality to the first term on the right-hand side of \eqref{eq:MoserStep2}. Using exponents \(\left(p/(p-1),\, p\right)\), we get
\begin{equation}\label{eq:MoserStep3}
	\begin{split}
		&[p(s-1)+1] \int_{\mathcal{A}_1} |\nabla \hat u|^p  \eta^p \hat u^{p(s-1)}   \, dx \\&
		\leq   \varepsilon [p(s-1)+1] \int_{\mathcal{A}_1} |\nabla \hat u|^p \eta^p \hat u^{p(s-1)} \, dx + {C}_{\varepsilon}(p,s) \int_{\mathcal{A}_1} \hat u^{ps} |\nabla \eta|^p  \, dx \\
		& \quad+ R^p\|c\|_{\infty}\int_{\mathcal{A}_1} 	\hat u^{q+p(s-1)+1} \eta^p \, dx,
	\end{split}
\end{equation}
where ${C}_{\varepsilon}(p,s)$ is a  positive constant.
Hence  
\begin{equation}\label{eq:MoserStep4}
\begin{split}
	& (1-\varepsilon)[p(s-1)+1] \int_{\mathcal{A}_1} |\eta \hat u^{s-1}\nabla \hat u|^p     \, dx \\
	\leq&  {C}_{\varepsilon}(p,s) \int_{\mathcal{A}_1}  |\nabla \eta|^p \hat u^{ps}  \, dx  + R^p\|c\|_{\infty}\int_{\mathcal{A}_1} 	\hat u^{q+p(s-1)+1} \eta^p \, dx.
		\end{split}
	\end{equation}
Now, noting that, for $p>1$, there exists a positive constant $C_p$ such that 
\begin{equation*}
\begin{split}
|\nabla (\eta \hat u^s)|^p\leq C_p (s^p|\eta \hat u^{s-1}\nabla \hat u|^p+|\hat u^s\nabla \eta|^p),   
\end{split}
\end{equation*}
 we deduce that the following inequality holds:
\begin{equation}\label{eq:lemmaTypeYoungBis1}
\begin{split}
	\int_{\mathcal{A}_1} |\eta \hat u^{s-1}\nabla \hat u|^p  \,dx & \geq \frac{1}{C_p} \cdot \frac{1}{s^p} \int_{\mathcal{A}_1} |\nabla (\eta \hat u^s)|^p \,dx - \frac{1}{s^p} \int_{\mathcal{A}_1} |\hat u^s \nabla \eta|^p\, dx.
\end{split}
\end{equation}

By \eqref{eq:lemmaTypeYoungBis1}, we get 
\begin{equation}\label{eq:MoserStep5}
	\begin{split}
		&\frac{1-\varepsilon}{C_p} \cdot \frac{p(s-1)+1}{s^p} \int_{\mathcal{A}_1}  |\nabla (\eta \hat u^{s} )|^p     \, dx\\
		\leq& {C}(\varepsilon,p,s) \int_{\mathcal{A}_1}  |\hat u^{s} \nabla \eta |^p \, dx +  {C}_{\varepsilon}(p,s) \int_{\mathcal{A}_1}  |\nabla \eta|^p \hat u^{ps}  \, dx  \\
		&+  R^p\|c\|_{\infty}\int_{\mathcal{A}_1} 	\hat u^{q-p+1} \eta^p \hat u^{ps} \, dx\\
		\leq&   \hat{\mathcal{C}} \int_{\mathcal{A}_1}|\nabla \eta|^p  \hat u^{ps}  \, dx  + R^p\|c\|_{\infty}\int_{\mathcal{A}_1} \hat u^{q-p+1} (\eta \hat u^s)^p \, dx,
	\end{split}
\end{equation}
%
%
%
where $\hat{\mathcal{C}}$ depends on $ \varepsilon, p, s$. Using  H\"older's inequality with exponents $(p^*/p,N/p)$ in the right hand side of \eqref{eq:MoserStep5}, it follows that
\begin{equation}\label{eq:MoserStep7}
	\begin{split}
		&\frac{1-\varepsilon}{C_p} \cdot \frac{p(s-1)+1}{s^p} \int_{\mathcal{A}_1}  |\nabla (\eta \hat u^{s} )|^p     \, dx\\
		&\leq \hat{\mathcal{C}} \int_{\mathcal{A}_1}|\nabla \eta|^p  \hat u^{ps}  \, dx  +\|c\|_{\infty} \left(\int_{\mathcal{A}_1}  \hat u^{(q-p+1)\frac{N}{p}}	R^N \,dx\right)^{\frac{p^*-p}{p^*}} \left(\int_{\mathcal{A}_1} (\eta \hat u^s)^{p \chi}\, dx\right)^{\frac 1\chi}
        \\&=\hat{\mathcal{C}} \int_{\mathcal{A}_1}|\nabla \eta|^p  \hat u^{ps}  \, dx  +\|c\|_{\infty}  \|u^{(q-p+1)}\|_{L^{\frac Np}(\mathcal{A}_R)} \left(\int_{\mathcal{A}_1} (\eta \hat u^s)^{p \chi}\, dx\right)^{\frac{1}{\chi}},
	\end{split}
\end{equation}
where $\chi=p^*/p$. 
Since $s>(p-1)/p$, for $\varepsilon$ small enough, by \eqref{eq:MoserStep7} we get
\begin{equation}\label{eq:MoserStep9}
	\bar{\mathcal{C}} \int_{\mathcal{A}_1}  |\nabla (\eta \hat u^{s} )|^p  \, dx \leq  \hat{\mathcal{C}} \int_{\mathcal{A}_1}|\nabla \eta|^p  \hat u^{ps}  \, dx  + \|c\|_{\infty}\|u^{(q-p+1)}\|_{L^{\frac Np}(\mathcal{A}_R)} \left(\int_{\mathcal{A}_1} (\eta \hat u^s)^{p \chi}\, dx\right)^{\frac{1}{\chi}},
\end{equation}
where $\bar{\mathcal{C}}$ is a positive constant depending on $\varepsilon,p,s$. 
In conclusion, by  Sobolev inequality we have
\begin{equation}\label{eq:MoserStep10}
\begin{split}
	 &\bar{\mathcal{C}} \mathcal{C}_S^p\left(\int_{\mathcal{A}_1} (\eta \hat u^s)^{p \chi}\, dx \right)^{\frac{1}{\chi}}  \leq   \bar{\mathcal{C}} \int_{\mathcal{A}_1} |\nabla (\eta \hat u^{s} )|^p  \, dx   \\
	\leq &  \hat{\mathcal{C}} \int_{\mathcal{A}_1}|\nabla \eta|^p  \hat u^{ps}  \, dx    + \|c\|_{\infty} \|u^{(q-p+1)}\|_{L^{\frac Np}(\mathcal{A}_R)} \left(\int_{\mathcal{A}_1} (\eta \hat u^s)^{p \chi}\, dx \right)^{\frac{1}{\chi}}.
\end{split}
\end{equation}
To implement Moser's iteration scheme, we rewrite the last inequality in the form
\begin{equation}\label{eq:MoserStep11}
	\left(\int_{\mathcal{A}_1} (\eta \hat u^s)^{p \chi}\, dx \right)^{\frac{1}{\chi}} \leq  \mathcal{C}_1 \int_{\mathcal{A}_1}|\nabla \eta|^p  \hat u^{ps}  \, dx  +  \mathcal{C}_2 \|u^{(q-p+1)}\|_{L^{\frac Np}(\mathcal{A}_R)} \left(\int_{\mathcal{A}_1} (\eta \hat u^s)^{p \chi}\, dx \right)^{\frac{1}{\chi}},
\end{equation}
where $\mathcal{C}_1:=\mathcal{\hat C}/( \bar{\mathcal{C}} \mathcal{C}_S^p)$,  $\mathcal{C}_2:=\|c\|_{\infty}/( \bar{\mathcal{C}} \mathcal{C}_S^p)$.

For $t\in (\beta,+\infty)$ as in the statement of Lemma \ref{aux2}, let $k\in \N$ so that $\beta \chi^k\le t\le \beta \chi ^{k+1}$. Then there exist positive constants $\mathcal C_1$ and $\mathcal C_2$ such that \eqref{eq:MoserStep11} holds for all $(p-1)/p< s \leq \chi ^{k+1}$.  \\
Now, if $\|c\|_{\infty}\neq 0$, we set  $\sigma=1 / (2\mathcal{C}_2)$, and choosing $\bar R$ sufficiently large such that \eqref{eq:assumptionsmallaux2} holds and $R > 8 \bar R$, we get
\begin{equation}\label{eccola}
	\left(\int_{\mathcal{A}_1} (\eta \hat u^s)^{p \chi}\, dx \right)^{\frac{1}{\chi}} \leq  \mathcal{C} \int_{\mathcal{A}_1}|\nabla \eta|^p  \hat u^{ps}  \, dx  
\end{equation}
for all $(p-1)/p <s \leq \chi ^{k+1}$ and $\mathcal C=\mathcal C_1/2.$  \ \\
Now by choosing appropriate functions $\eta$ (e.g.~\cite[pg.~285]{serrin}),  we want to apply \emph{finite} $(k+1)$-Moser's iterations based on the steps $\beta, \beta \chi, \ldots, \beta\chi^{k+1}$, to conclude the proof (see, e.g., \cite{HanLin,Xiang} for further details), \\ 
To this end, estimate \eqref{eccola} suggests taking $s=\beta/p$ as the initial step.  Observe that, since 
$$\frac{N(p-1)}{N-p} < q < p^*-1,$$ 
the corresponding exponent $s$ satisfies $(p-1)/p<s<\chi$.\\
After computating of $(k+1)$-Moser's iterations,  being $t\in (\beta,+\infty)$ and $\beta \chi^k\le t\le \beta \chi ^{k+1}$ we obtain
 \begin{equation}\label{abc}
	\left(\fint_{\mathcal{D}_1}  |\hat u|^t\, dx \right)^{\frac{1}{t}} \leq  \mathcal{C} \left(\fint_{\mathcal{A}_1}| \hat u|^{\beta} \, dx \right) ^{\frac{1}{\beta}}. 
\end{equation}
This proves \eqref{eq:EstAt0aux2}. We conclude by noting that if $c(x)= 0$ a.e.~in $\R^N$, then the same conclusion holds true for any $\sigma>0$.
\end{proof}

Let us prove the following:
\begin{thm}\label{teoo}
	Let \(c(x)\in L^{\infty}(\mathbb{R}^N\setminus B_{R_1})\) be a nonnegative 
function, and let \(u\) be a weak solution of \eqref{eq:Probexterior}, 
satisfying the assumption \eqref{ipotesiu}. Then there exists a positive constant $C=C(N,p,u, q, \|c\|_{\infty})$ such that
	$$|u(x)|\leq \frac{C}{{|x|^{\frac{p}{q-p+1}+\sigma_2}}} \qquad \text{in } {B}_{R_2}^c,$$
	where  $\sigma_2$ is given in Lemma \ref{lem:aux1} and  $ R_2> 8 \bar R$ is a constant depending on $N, p, q, \|c\|_{\infty}$, $u$.
\end{thm}

\begin{proof}

Let us fix $t >\beta$ as in Lemma \ref{aux2} and $\kappa := \min\{\tau, \sigma\}$, where $\tau$ and $\sigma$  as in Lemma \ref{lem:aux1} and Lemma \ref{aux2}. Let $\bar R > R_1$ be such that \eqref{eq:assumptionsmall} holds for $\kappa$ and consider 
$\hat u(x)=u(Rx)$ and $\hat c(x)=c(Rx)$, for $R>R_1$ fixed. We note that $\hat u$ satisfies the equation
$$-\Delta_p \hat u -V_R(x) \hat u^{p-1}=0 \qquad \text{in } \mathcal{D}_1,$$
where $$V_R(x):=R^p \hat c(x) \hat u ^{q-p+1}(x).$$
Hence, as in the proof of \cite[Theorem $1$]{serrin} a classical Moser iteration argument yields 
\begin{equation}\label{eq:MoserHanLin}
	\sup_{{B}_r(x)} \hat u \leq \mathbf{C} \left(\fint_{{B}_{2r}(x)} \hat u^p \, dx\right)^\frac{1}{p}
\end{equation}
for any ball ${B}_{2r}(x) \subset \subset \mathcal{D}_1$, where $\mathbf{C}=\mathbf{C}(N,p,\|V_R\|_{L^\alpha(\mathcal D_1)})$ and $\alpha:=t/(q-p+1)>N/p$. We claim that $\|V_R\|_{L^\alpha(\mathcal D_1)}$ is uniformly bounded with respect to $R$. Indeed from Lemma \ref{aux2}, since 
$$-\frac{N}{\alpha}+\frac{N}{t}(q-p+1)=0$$
we get
\begin{equation}\label{cnondipendedaR}
    \begin{split}
        \left( \int_{\mathcal D_1}V_R^\alpha \, dx\right)^{\frac{1}{\alpha}}& = R^{p-\frac{N}{\alpha}}\left( \int_{\mathcal D_R}  c^\alpha(x)u^{(q-p+1)\alpha} \, dx\right)^{\frac{1}{\alpha}} \\ &\le \mathcal CR^{p-\frac{N}{\alpha}+N\frac{q-p+1}{t}}\left( \fint_{\mathcal D_R}  u^{t} \, dx\right)^{\frac{q-p+1}{t}} \\ &\le \mathcal C R^{p-\frac{N}{\alpha}+N\frac{q-p+1}{t}-p}\left( \int_{\mathcal A_R}  u^{\frac{(q-p+1)N}{p}} \, dx \right)^{\frac{p}{N}}  \qquad \forall  R\ge 8 \bar R,
    \end{split}
\end{equation}
where $\mathcal C$ is a positive constant depending on $N, p, \alpha, \|c\|_{\infty}$ and $\bar R$.

Using a covering argument, we deduce that
\begin{equation}\label{eq:MoserHanLin2}
	\sup_{{B}_2 \setminus {B}_1} \hat u \leq \mathbf{C} \left(\fint_{\mathcal{D}_1} \hat u^p \, dx\right)^\frac{1}{p}
\end{equation}
Noticing that $\hat u(x)= u(Rx)$, by \eqref{eq:MoserHanLin2} we obtain that
\begin{equation}\label{eq:MoserHanLin3}
	\sup_{{B}_{2R} \setminus {B}_R} u \leq \mathbf{C} \left(\fint_{\mathcal{D}_R}  u^p \, dx\right)^\frac{1}{p}
\end{equation}
for each  $R > \max 8 \bar R $. By applying the H\"older's inequality in \eqref{eq:MoserHanLin3}, we get
\begin{equation}\label{eq:MoserHanLin4}
	\sup_{{B}_{2R} \setminus {B}_R} u \leq \mathbf{C} \left(\fint_{\mathcal{D}_R}  u^p \, dx \right)^\frac{1}{p} \leq \mathbf{C} \left(\fint_{\mathcal{D}_R}  u^{p^*} \, dx\right)^\frac{1}{p^*},
\end{equation}
for each $R > 8 \bar R$ and $\mathbf{C}$ depends only on $N, p, \alpha, \bar R, \|c\|_{\infty}$.
By \eqref{eq:EstAt0aux2}-Lemma \ref{aux2}, we deduce 
\begin{equation}\label{Linno}
    \sup_{{B}_{2R} \setminus {B}_R} u \leq \mathbf{C} \left(\fint_{\mathcal{A}_R} u^{(q-p+1)\frac{N}{p}} \, dx \right)^{\frac{p}{(q-p+1)N}}=\mathbf{C}\|u^{q-p+1}\|_{L^{\frac Np}(\mathcal{A}_R)}R^{-\frac{p}{q-p+1}}.
\end{equation}

for each $R >  \bar R$ and $\mathbf{C}$ depends only on $N, p, \alpha, \bar R, \|c\|_{\infty}$.
Now we note that, since $\mathcal{A}_R \subset {B}_{1/\bar R}^c$ for any $R > 8 \bar R$, there exist, by Lemma \ref{lem:aux1}, $ \sigma_2>0$ depending only on $N, p$ such that
$$\|u^{q-p+1}\|_{L^{\frac Np}({B}_R^c)} \leq \frac{\mathcal{C}}{R^{\sigma_2}}, \quad \text{for } R \geq 8 \bar R.$$ 
Setting  $R_2> 8 \bar R$, then we get the thesis.
\end{proof}

\begin{rem}\label{ustainLinfinitolontanodallorigine}
    Let $u$ be a weak solution of \eqref{eq:Probexterior}.  Let \(R>R_1\) be fixed, then by Theorem \ref{teoo}, it follows that the solution \(u \in L^\infty(\mathbb{R}^N \setminus B_R)\). 
Indeed, let \(x_0 \in \mathbb{R}^N\) such that \(B_{2r}(x_0) \subset\subset \mathbb{R}^N \setminus B_{R_2}\), with $R_2$ as in Theorem \ref{teoo}. 
Then, by \eqref{Linno}, we deduce
\[
\sup_{B_r(x_0)} u 
\leq \mathbf{C} \left(\fint_{B_{2r}(x_0)} u^{(q-p+1)\frac{N}{p}} \, dx \right)^{\frac{p}{(q-p+1)N}},
\]
with $\mathbf{C}$ positive constant. Now, letting the ball \(B_r(x_0)\) move in \(\mathbb{R}^N\) away from the origin, and using the fact that 
\(u^{q-p+1} \in L^{\frac Np}(\mathbb{R}^N \setminus B_{R})\), we obtain that \(u \in L^\infty(\mathbb{R}^N \setminus B_{R_{2}})\).
To obtain the \(L^\infty\)-bound in $B_{R_{2}} \setminus B_{R}$, it is enough to invoke Remark \ref{primorem}.  Again exploting Remark \ref{primorem},  one proves that  if $I_\Gamma$ a neighborhood of $\Gamma$ such that $d(\Gamma, \partial I_\Gamma)>0$,  $u \in L^{\infty}(\mathbb{R}^N\setminus I_\Gamma)$.\\
Moreover, by using an approach similar to the one employed in the proof of 
Lemma~\ref{lem:aux1}, we deduce that 
\(u \in \mathcal{D}^{1,p}(B_{R}^{\,c})\).  From this, since $u\in W_{loc}^{1,p}(\mathbb{R}^N\setminus\Gamma)$,  $u \in \mathcal{D}^{1,p}(\mathbb{R}^N\setminus I_\Gamma)$.
\end{rem}

The next result is devoted to show the existence of some special supersolutions of our problem, in order to perform a comparison between them and the solutions of our subcritical singular problem.

\begin{prop} \label{prop:SubSuperComp}
	Given two constants $A>0$ and $\alpha > p$, there exist $\varepsilon=\varepsilon (\alpha,p)>0$ and $0<\delta=\delta (N,p,\alpha,A) <1$  such that
	\begin{equation}v(|x|)=\frac{1-\delta |x|^{-\varepsilon}}{|x|^{\frac{N-p}{p-1}}} \in \mathcal{D}^{1,p}(B^{c}_{\tilde R})\end{equation}
	is a positive supersolution to the equation
	\begin{equation}-\Delta_p v = g(x)v^{p-1} \quad \text{in } B^{c}_{\tilde R},
\end{equation}
	for some positive constant $\tilde R:=\tilde R(N,A,p,\alpha,\delta,\varepsilon)>1$,  where $g(x)$ is a positive function in $L^{\frac{N}{p}}(B^{c}_{\tilde R})$ such that
	\begin{equation}\label{g2}
 g(x) \geq A|x|^{-\alpha}\quad \text{in } B^{c}_{\tilde R}.
 \end{equation}
\end{prop}

\begin{proof} Let us consider $\delta, \varepsilon>0$ and let us define the function
$$u(x)=v(|x|)=\frac{1-\delta |x|^{-\varepsilon}}{|x|^{\frac{N-p}{p-1}}}.$$
Let us restrict to those $\delta$ and $\varepsilon$ be such that $\delta |x|^{-\varepsilon}<1$.  A direct computation shows that $v$ satisfies
\begin{equation} \label{eq:computedSuperBDD2}
	\begin{split}
		-\Delta_p v = g(x) |v|^{p-2} v,
	\end{split} 
\end{equation}
where
$$g(x):= \frac{h(\delta |x|^{-\varepsilon})}{(1 -\delta |x|^{-\varepsilon})^{p-1} |x|^p},$$
and
\begin{equation}
\begin{split}
h(t):=t \varepsilon (p-1)\left(\frac{N-p}{p-1}+\varepsilon\right) \left|\frac{N-p}{p-1} - \left(\frac{N-p}{p-1} + \varepsilon\right) t \right|^{p-2}>0, 
\quad\text{ with } t :=\delta |x|^{-\varepsilon}.
\end{split}
\end{equation}

Noticing that $h(0)=0$ and $\displaystyle h'(0)>0$, there exist $0<\delta_h<1$ such that \begin{equation}\label{basta}
0<\frac{1}{2}h'(0)t\le h(t)\le 2h'(0)t \qquad \forall t\in(0,\delta_h).
\end{equation}
Let us choose $\displaystyle \varepsilon=\frac{\alpha-p}{2}>0$, $\delta=\min \{\delta_h, 1/2\}$; to get \eqref{g2}, taking into account the initial position $\delta |x|^{-\varepsilon}<1$, it is enough to choose 
$$\tilde R\geq\max \left\{1,\left(\frac{2A}{\delta h'(0)}\right)^{\frac{2}{\alpha-p}}\right\}.$$ 
As a consequence,  $v\in  \mathcal{D}^{1,p}(B^{c}_{\tilde R})$ and $g\in L^{\frac{N}{p}}(B^{c}_{\tilde R})$.
\end{proof}

Let us consider the following equation
\begin{equation}\label{eq:comparison}
	-\Delta_p w =f(x) w^{p-1} \qquad \text{in } \Omega,
\end{equation}
where $\Omega$ is an open subset of $\R^N$, $w>0$ and $w \in \mathcal{D}^{1,p}(\Omega)$. In the applications, the domain \(\Omega\) will be an exterior domain,  namely \(\Omega = B_R^{\,c}\) for some sufficiently large \(R\).

Now, we state a comparison principle whose proof can be found in \cite[Theorem~3.3]{Xiang} and \cite[Proposition~3.7]{EMSV}, which is based on a careful use of the Picone identity.

\begin{prop}\label{prop:Comparison3}
Let $\Omega$ be an exterior domain such that $\R^N \setminus \Omega$ is bounded and $f\in L^{\frac{N}{p}}(\Omega)$. Let $u \in \mathcal D^{1,p}(\Omega) $ be a weak positive subsolutions to \eqref{eq:comparison} and $v\in \mathcal D^{1,p}(\Omega)$ be a positive supersolution of
\begin{equation}\label{eq:comparison3}
	-\Delta_p v =g(x) v^{p-1} \qquad \text{in } \Omega,
\end{equation} 
with $g\in L^{\frac{N}{p}}(\Omega)$.
Assume that $f\le g$ in $\Omega$.
If $u \leq v$ on $\partial \Omega$ and

\begin{equation}\label{logAss}
	\limsup_{R \rightarrow + \infty} \frac{1}{R} \int_{{B}_{2R} \setminus {B}_R} u^p |\nabla \log v|^{p-1}=0,
\end{equation}
then $$u \leq v \quad \text{in } \Omega.$$ 
\end{prop}

%
%

\section{Proof of the asymptotic estimates
} \label{AsympEst}

This section is dedicated to the proof of the decay estimates of the solution to \eqref{eq:Probexterior} and its gradient.

\begin{thm} \label{thm:asymptEst}
	Let \(c(x)\in L^{\infty}(\mathbb{R}^N\setminus B_{R_1})\) be a nonnegative 
function, and let \(u\) be a weak solution of \eqref{eq:Probexterior}, 
satisfying the assumption \eqref{ipotesiu}.
 Then there exists a positive constants $\hat R>R_1$ depending on $N, p,\|c\|_{\infty},q$ and  $u$, such that
	\begin{equation}\label{eq:estatInf}
		\frac{\hat c}{|x|^{\frac{N-p}{p-1}}}  \leq u(x) \leq \frac{\hat C}{|x|^{\frac{N-p}{p-1}}} \qquad x \in B_{\hat R}^c,
	\end{equation}
where $\hat c$, $\hat C$ are positive constants depending on $N,p, q, \|c\|_{\infty}$ and $u$.
\end{thm}
\begin{proof}[Proof of Theorem \ref{thm:asymptEst}]
We start by proving the estimate from above in \eqref{eq:estatInf}. Let us consider $u$ a solution of \eqref{eq:Probexterior}, that we can rewrite as follows
\begin{equation}\label{22}
    -\Delta_p u = c(x)u^{q-p+1}u^{p-1} \qquad \text{in } {B}_{R_1}^c,
    \end{equation}
  Denoting by $f(x):=c(x)u^{q-p+1}\in L^{\frac{N}{p}}({B}_{R_1}^c)$, by Theorem \ref{teoo} we have $|f(x)|\le C|x|^{-\alpha}$ with $\alpha=\left(\frac{p}{q-p+1}+\sigma_2\right)(q-p+1)$ for $x\in {B}_{R_2}^c$, where $C=C(N,p,\|c\|_\infty,u,q)>0$ and $R_2$ is given in  Theorem \ref{teoo}.
    Since $\alpha >p$, given the above $C>0$, by Proposition \ref{prop:SubSuperComp}, there exist $0<\delta<1$, $\varepsilon >0$ such that for a suitable $\tilde R:=\tilde R(N,p,\delta,\varepsilon)> 1$, the function 	
$$v(|x|)=\frac{1-\delta |x|^{-\varepsilon}}{|x|^{\frac{N-p}{p-1}}} \in \mathcal{D}^{1,p}({B}_{\tilde R}^c)$$ 
is a positive supersolution of \eqref{eq:comparison3} in ${B}_{\tilde R}^c$, with $g\in L^{\frac{N}{p}}({B}_{\tilde R}^c)$ satisfying $g(x)\ge C |x|^{-\alpha}$ on $B_{\tilde R}^c$.

Now, let us choose $\hat R:= \max\{\tilde R, R_2\}$. Note that $f\leq g$ on $B_{\hat R}^c$. Let us consider $\Upsilon >0$, $M= \sup_{\partial {B}_{\hat R}} u$, $N= \sup_{\partial{B}_{\hat R}} 1/v$ and define
$$w(x):= N(M+\Upsilon)v(|x|).$$
We note that $w$ is a positive supersolution of \eqref{eq:comparison3} in ${B}_{\hat R}^c$ and $\inf_{\partial{B}_{\hat R}} w = M+\Upsilon>0$ and $u \leq w$ on $\partial {B}_{\hat R}$. 
We verify the condition \eqref{logAss}. Since $|\nabla \log v(x)|\le C|x|^{-1}$, by H\"older's inequality with exponents $(p^*/p,N/p)$, we have 
\begin{equation}
    \begin{split}
      	&\limsup_{R \rightarrow + \infty} \frac{1}{R} \int_{{B}_{2R} \setminus {B}_R} u^p |\nabla \log w|^{p-1}\,dx
        \\&\leq \limsup_{R \rightarrow + \infty}\frac{C}{R^p}\int_{{B}_{2R} \setminus {B}_R} u^p\,dx\le C\limsup_{R \rightarrow + \infty}\left(\int_{B_{2R}\setminus B_R}|u|^{p^*}\,dx\right)^{\frac{p}{p^*}}=0,  
    \end{split}
\end{equation}
where $C$ is a constant independent of $R$. Hence, by Proposition \ref{prop:Comparison3} we deduce that
$$u \leq w \qquad \text{in } {B}_{\hat R}^c.$$
Passing to the limit for $\Upsilon \rightarrow 0^+$ we obtain that
$$u(x) \leq \frac{\hat C}{|x|^{\frac{N-p}{p-1}}} \qquad \text{in } {B}_{\hat R}^c,$$
where $\hat C=\hat C(\hat R,u,N,p)=M \cdot N$.

We conclude by proving the estimate from below in \eqref{eq:estatInf}.  Let \(u\) be a weak solution of \eqref{eq:Probexterior}. 
Then \(u\) is superharmonic in the exterior of a ball of radius \(\hat R\), 
with \(\hat R\) defined above.

We claim that 
\begin{equation}\label{supersolution}
\int_{B_{2R}\setminus B_R}|\nabla \log u|^p \le CR^{N-p},
\end{equation}
for $R$ sufficiently large and some positive constant $C$. Indeed,
we have \begin{equation}\label{zio}
\int_{\R^N\setminus B_{\hat R}}|\nabla u|^{p-2}\langle \nabla u,\nabla \eta\rangle\,dx\geq 0\quad \forall \varphi\in C^1_c(\R^N\setminus B_{\hat R}),
\end{equation}
with $\eta$ nonnegative function. Therefore, considering the test function $\varphi =\zeta ^p u^{1-p}$, with $\zeta \in C^\infty_c(\R^N \setminus B_{\hat R})$ be nonnegative function, and testing in \eqref{zio} we have 
\begin{equation}\label{zioo}
    \begin{split}
        (p-1)\int_{\R^N \setminus B_{\hat R}} |\nabla u|^p \zeta ^pu^{-p}\,dx\leq p\int_{\R^N \setminus B_{\hat R}} |\nabla u|^{p-1}|\nabla \zeta|\zeta^{p-1} u^{-(p-1)}\,dx.
    \end{split}
\end{equation}
By H\"older's inequality we get 
\begin{equation}\label{aa}
  \int_{\R^N \setminus B_{\hat R}} |\nabla \log u|^p \zeta ^p\,dx  \le \frac{p}{p-1}\left(\int_{\R^N \setminus B_{\hat R}}|\nabla \log u|^p \zeta ^p\,dx \right)^{\frac{p-1}{p}}\left(\int_{\R^N \setminus B_{\hat R}}|\nabla \zeta |^p\,dx\right)^{\frac{1}{p}}.
\end{equation}
We fix \(R/2 >\hat R\) and we consider the test function \(\zeta = \psi_{2R} \eta\), where \(\psi_{2R}\) and 
\(\eta\) are defined respectively in \eqref{eq:cutoffR} and \eqref{eq:cutoff}; moreover $\text{supp}(\eta) \subset \R^N \setminus B_{R/2}$ and $\text{supp}(\psi_{2R}) \subset B_{4R}$.  The second therm in the right hand side of \eqref{aa} becomes
\begin{equation}\label{controllo}
    \left(\int_{\R^N \setminus B_{\hat R}}|\nabla \zeta |^p\,dx\right)^{\frac{1}{p}} \leq \left(\int_{B_R \setminus B_{R/2}}|\nabla \eta|^p \,dx\right)^{\frac{1}{p}} + \left(\int_{B_{4R} \setminus B_{2R}}|\nabla \psi_{2R} |^p \,dx\right)^{\frac{1}{p}} \leq  C R^\frac{N-p}{p}.
\end{equation}

By \eqref{aa} and \eqref{controllo}, we obtain the claim \eqref{supersolution}. 

Now we set $\displaystyle c_2=\inf_{\partial B_{\hat R}} u>0$. We note that $v(x)={c_2}\hat R^{\frac{N-p}{p-1}}|x|^{-\frac{N-p}{p-1}}$ is a $p$-harmonic function. Moreover the condition \eqref{logAss} is verified; indeed by H\"older's inequality and \eqref{supersolution} we have 
\begin{equation}\label{oo}
\begin{split}
    \frac{1}{R} \int_{B_{2R}\setminus B_R} v^p|\nabla \log u|^{p-1}\,dx\le C R^{-1-p\frac{N-p}{p-1}+\frac{N}{p}}\left(\int_{B_{2R}\setminus B_R}|\nabla \log u|^p\,dx\right)^{\frac{p-1}{p}} \\ \le C R^{-1-p\frac{N-p}{p-1}+\frac{p-1}{p}(N-p)+\frac{N}{p}}=CR^{\frac{-p+N}{p-1}} \rightarrow 0,
\end{split}
\end{equation}
since $p<N$. Applying Proposition \ref{prop:Comparison3}, we conclude that 
$$
u(x)\ge v(x)=\hat R^{\frac{N-p}{p-1}}\frac{c_2}{|x|^{\frac{N-p}{p-1}}}\quad \text{in } {B}_{\hat R}^c
$$ 
and therefore the thesis, with $\hat c=R^{\frac{N-p}{p-1}}c_2$.
\end{proof}

Now we recall a classification result proved in \cite{Sciu16} that will be essential to prove the asymptotic behavior of the gradient of solutions to \eqref{eq:Probexterior}. 
\begin{thm}[\cite{Sciu16}, Theorem 2.1]\label{thmarmonic}
Let $v\in C^{1,\alpha}_{loc}(\mathbb{R}^N\setminus\{0\})$ be $p$-harmonic in $\mathbb{R}^N\setminus\{0\}$ and assume that
\begin{equation}\label{limarm}
\underset{|x|\rightarrow \infty}{\lim}\, v(x)=0\qquad\text{and}\qquad\underset{|x|\rightarrow 0}{\lim}\, v(x)=\infty\,.
\end{equation}
Then 
\begin{equation}
v(x)\,=\, \frac{\beta}{|x|^{\frac{N-p}{p-1}}}\qquad \text{for some}\,\, \beta \in\mathbb{R}^+\,.
\end{equation}
\end{thm}

At this point we are ready to prove the gradient estimates.

\begin{thm}\label{stimagradiente}
    Let \(c(x)\in L^{\infty}(\mathbb{R}^N\setminus B_{R_1})\) be a nonnegative 
function, and let \(u\) be a weak solution of \eqref{eq:Probexterior}, 
satisfying the assumption \eqref{ipotesiu}.
 Then there exist positive constants $\tilde c$, $\tilde C$ depending on $N,p,\|c\|_{\infty},q $ and $u$ such that
	\begin{equation}\label{eq:estatInfgrad}
		\frac{\tilde c}{|x|^{\frac{N-p}{p-1}+1}}  \leq |\nabla u(x)| \leq \frac{\tilde C}{|x|^{\frac{N-p}{p-1}+1}} \qquad x \in B_{\hat R}^c,
	\end{equation}
where $\hat R$ is a constant depending on $N, p,\|c\|_{\infty},q $ and  $u$.
\end{thm}

\begin{rem}
    The following proof is inspired by \cite{Sciu16}, but for the sake of completeness, we include its details.
\end{rem}

\begin{proof}[Proof of Theorem \ref{stimagradiente}]
We start by showing the estimate from above.  For $R_n$ tending to infinity, let us consider 
\begin{equation}\label{riscalata}
w_n(x):=R_n^{\frac{N-p}{p-1}}u(R_nx) \ \ \text{ and } \ \ c_n(x):=c(R_nx) \qquad \forall x\in  {B}_{2} \setminus {B}_{1/2}.  
\end{equation}
We remark that, by Theorem \ref{thm:asymptEst}, $w_n$ is uniformly bounded in $L^{\infty}({B}_{2} \setminus {B}_{1/2})$ and it weakly solves 
\begin{equation}
    -\Delta_pw_n = R_n^{N-\frac{N-p}{p-1}q} c_n(x) w_n^q \qquad \text{in } \R^N\setminus B_{\frac{R_1}{R_n}}.
\end{equation}

Recalling that $q> N(p-1)/(N-p)$, we note that $N-\frac{N-p}{p-1}q<0$. By \cite{L,L2}, $w_n$ is also uniformly bounded in $C^{1,\alpha}(\mathcal K),$ for $0<\alpha<1$ and for any compact set $\mathcal K\subset {B}_{2} \setminus {B}_{1/2} $. Hence, for $R_n$ sufficiently large we get the estimate from above in \eqref{eq:estatInfgrad}. 

Now we prove the estimate from below. Suppose by contradiction that there exist a sequence of points $\{x_n\}$ such that 
\begin{equation}\label{bau}
|x_n|^{\frac{N-p}{p-1}+1}|\nabla u(x_n)|\rightarrow 0 \qquad \text{for } |x_n|\rightarrow \infty.
\end{equation}
For $0<a<A$ fixed, let us consider 
$$w_n(x):=R_n^{\frac{N-p}{p-1}}u(R_nx).$$ For $n$ sufficiently large, from Theorem \ref{thm:asymptEst}, up to relabel the constants, we have 
\begin{equation}\label{eq:diagonal}
\frac{\hat c}{A^{\frac{N-p}{p-1}}}\le w_n(x)\le \frac{\hat C}{a^{\frac{N-p}{p-1}}}\qquad \text{in } \overline{{B}_{A}\setminus {B}_{a}},
\end{equation}
and in particular
\begin{equation}\label{eq:diagonal2}
\begin{split}
\frac{\hat c}{A^{\frac{N-p}{p-1}}}\le w_n(x)\le \frac{\hat C}{A^{\frac{N-p}{p-1}}}\qquad \text{on } \partial {B}_{A}, 
 \quad \frac{\hat c}{a^{\frac{N-p}{p-1}}}\le w_n(x)\le \frac{\hat C}{a^{\frac{N-p}{p-1}}}\qquad \text{on } \partial {B}_{a}.
\end{split}
\end{equation}
Furthermore, recalling the estimate from above of the gradients of the weak solution $u$, proved previously, we get
\begin{equation}
    |\nabla w_n(x)|\le \frac{\tilde C}{a^{\frac{N-p}{p-1}+1}} \quad\text{in } \overline{{B}_{A} \setminus {B}_{a} }.
\end{equation}
For $a,A$ fixed, by \cite{L,L2}, $w_n$ is also uniformly bounded in $C^{1,\alpha}(\mathcal K),$ for $0<\alpha<1$ and for any compact set $\mathcal K\subset {B}_{A}\setminus {B}_{a}$. Moreover 
$$w_n\rightarrow w_{a,A} \qquad \text{in } B_A\setminus B_a,$$ in the norm $C^{1,\alpha '}$, for $0<\alpha' <\alpha.$ 
Moreover, since $q> N(p-1)/(N-p)$ and $w_n$ weakly solves
\begin{equation}
-\Delta_p w_n = R_n^{N-\frac{N-p}{p-1}q} c_n(x) w_n^{q} \qquad \text{in } {B}_{A} \setminus {B}_{a},
\end{equation}
we deduce that 
\begin{equation}
-\Delta_pw_{a,A}=0 \qquad \text{in } {B}_{A}\setminus {B}_{a}.
\end{equation}
Now we take $a_j=1/j$ and $A_j=j$, for $j\in \N$ and we construct $w_{a_j,A_j}$ as above. For $j$ goes to infinity, using a standard diagonal process, we construct a limiting profile $w_{\infty}$ so that 
\begin{equation}
-\Delta_pw_{\infty}=0 \qquad \text{in } \R^N\setminus \{0\},
\end{equation}
with $w_{\infty}\equiv w_{a_j,A_j}$ in ${B}_{A_j} \setminus {B}_{a_j}$.

Since $w_{\infty}$ satisfies \eqref{eq:diagonal} and \eqref{eq:diagonal2}, it follows that
\[
\lim_{|x| \to \infty} w_{\infty}(x) = 0 \quad \text{and} \quad \lim_{|x| \to 0} w_{\infty}(x) = \infty,
\]
i.e. it satisfies the assumptions of Theorem \ref{thmarmonic}. Thus, we get
\begin{equation}
    w_{\infty}(x)= \frac{\beta}{|x|^{\frac{N-p}{p-1}}}.
\end{equation}
Now we set $y_n=x_n/R_n$, and by \eqref{bau}, we deduce that $|\nabla w_n(y_n)|$ tends to zero as $R_n$ tends to infinity. This fact and the uniform convergence of the gradients imply that there exist $\overline y\in \partial {B}_{1}$ such that
$$|\nabla w_{\infty}(\overline y)|=0.$$
This gives a contradiction with the fact that $w_{\infty}$ has no critical points.
\end{proof}

The proof of Theorem~\ref{thm:SolGradEst} easily follows combining Theorem \ref{thm:asymptEst} and \ref{stimagradiente}.

\section{Radial symmetry of the solutions} 
\label{sec:symmetry}

This section is devoted to the proof of Theorem~\ref{Teopmaggiore2} and Theorem~\ref{Teopminore2}. The section ends with the proof of Theorem~\ref{Teo:simmetriamenoilpunto}.

A key tool in our proofs is the moving plane technique together with the notion of 
$p$-capacity. In order to apply these ideas effectively, we introduce the following notation.

For \(\lambda <0\), we set
\[
T_\lambda = \{\, x \in \mathbb{R}^N : x_1 = \lambda \}, \qquad
\Sigma_\lambda = \{\, x \in \mathbb{R}^N : x_1 < \lambda \}.
\]
For any \(x = (x_1,x') \in \mathbb{R}^N\), we define the reflection of \(x\) 
through the hyperplane \(T_\lambda\) as
\[
x_\lambda = R_\lambda(x) := (2\lambda - x_1,\, x'), \qquad x' \in \mathbb{R}^{N-1}.
\]

\subsection{p-capacity and auxiliary cut-off functions}
We recall the definition of the \(p\)-capacity of a compact set \(A \subset \mathbb{R}^N\).  
For \(1 \le p \le N\), we set
\begin{equation}\label{p-capacita}
\operatorname{Cap}_p(A)
:= \inf \left\{
\int_{\mathbb{R}^n} |\nabla \varphi|^p \, dx < +\infty :
\varphi \in C_c^\infty(\mathbb{R}^n),\ \varphi \ge \chi_A
\right\}.
\end{equation}
Since \eqref{p-capacita} is invariant under reflections, it follows that
\begin{equation}\label{robbe}
\operatorname{Cap}_p(\Gamma) = \operatorname{Cap}_p(R_\lambda(\Gamma)).
\end{equation}

If  \(D \subset \mathbb{R}^N\) is a bounded domain, for any compact set \(A \subset D\) let us define,
\[
\operatorname{Cap}^D_p(A)
:= \inf \left\{
\int_D |\nabla \varphi|^p \, dx < +\infty :
\varphi \in C_c^\infty(D),\ \varphi \ge \chi_A
\right\}.
\]

Since by assumption \(\operatorname{Cap}_p(\Gamma)= 0\), by \eqref{robbe}, it follows that
\begin{equation}\label{capacitastramba}
\operatorname{Cap}^D_p(R_\lambda(\Gamma)) = 0, \quad \forall \; D \supset R_\lambda(\Gamma).
\end{equation}

Let \(\varepsilon>0\) be small, and let \(B^\lambda_\varepsilon\) denote the \(\varepsilon\)-neighborhood of \(R_\lambda(\Gamma)\).  
From  \eqref{capacitastramba}, there exists \(\varphi_\varepsilon \in C_c^\infty(B^\lambda_\varepsilon)\) such that  
\(\varphi_\varepsilon \ge 1\) on \(R_\lambda(\Gamma)\) and
\[
\int_{B^\lambda_\varepsilon} |\nabla \varphi_\varepsilon|^p \, dx < \varepsilon.
\]
Denoting by 
\(T : \mathbb{R} \to \mathbb{R}^+_0\) and \(g : \mathbb{R}^+_0 \to \mathbb{R}^+_0\) the functions
\begin{equation}\label{FunzioneTeG}
T(s) := \max\{0, \min\{s,1\}\} \mbox{ and }
g(s) := \max\{0, -2s + 1\},
\end{equation}
let us define
\begin{equation}\label{definizionecutoff}
\psi_\varepsilon(x) := g\big( T(\varphi_\varepsilon(x)) \big).
\end{equation}
Then \(\psi_\varepsilon(x)\in W^{1,p}_{loc}(\R^N)\) satisfies
\begin{itemize}
\item $\psi_\varepsilon = 1$ in  $\mathbb{R}^N \setminus B^\lambda_\varepsilon$,
\item $\psi_\varepsilon = 0 \ \text{in }B^\lambda_{\delta_\varepsilon}$,
where, $B^\lambda_{\delta_\varepsilon}$ is a $\delta_\varepsilon$-neighborhood of $R_\lambda(\Gamma)$, such that $\delta_\varepsilon < \varepsilon$.
\item\begin{equation}\label{Proprietacutoffdisingolareriflesso}
\int_{B^\lambda_\varepsilon} |\nabla \psi_\varepsilon|^p \, dx \le C \varepsilon,
\end{equation}
for some constant \(C>0\) independent of \(\varepsilon\). 
\end{itemize}
For further details on the notion of $p$‑capacity, we refer to \cite{EMS}.
 
Accordingly with the above notations, we define
\[
u_\lambda(x) := u(x_\lambda),\quad c_\lambda(x):=c(x_\lambda).
\]
Finally, we introduce the set
\[
\Lambda := \{\, \mu <0 : u \le u_\lambda \text{ in } \Sigma_\lambda\setminus R_\lambda (\Gamma) 
\text{ for every } \lambda \le \mu \,\},
\]
and, provided that \(\Lambda \neq \emptyset\), we set
\begin{equation}\label{lambdazero}
   \mu_0 := \sup \Lambda. 
\end{equation}

In the following, we exploit the fact that \(u_\lambda\) (in the sense of Definition~\ref{Definizione_debole}) 
is a solution of
\begin{equation}\label{ulambdasoddisfa}\tag{$\mathcal{P}_\Gamma^\lambda$}
\int_{\mathbb{R}^N} |\nabla u_\lambda|^{p-2} (\nabla u_\lambda, \nabla \varphi)\, dx
= \int_{\mathbb{R}^N} c_\lambda(x) u_\lambda^q\, \varphi \, dx
\qquad \forall\, \varphi \in C_c^{\infty}\big(\mathbb{R}^N \setminus R_\lambda(\Gamma)\big).
\end{equation}
We then set
\begin{equation}\label{definizionewlambda}
   w_\lambda(x) := (u - u_\lambda)(x), \qquad x \in \mathbb{R}^N \setminus (R_\lambda(\Gamma)\cup \Gamma). 
\end{equation}

In view of Remark \ref{ustainLinfinitolontanodallorigine}, we can state the following lemma, whose proof is standard.
\begin{lem} \label{test}
Let $u$ be a weak solution of \eqref{eq:Prob}, satisfying the assumption \eqref{ipotesiu}, and $\psi_\varepsilon$ defined as in \eqref{definizionecutoff}.  Let $f\in D^{1,p}(\mathbb{R}^N\setminus R_\lambda^{-1}(B^\lambda_{\delta_\varepsilon}))\cap L^\infty(\mathbb{R}^N\setminus R_\lambda^{-1}(B^\lambda_{\delta_\varepsilon}))$, such that $f^{q-p+1}\in L^{\frac Np}(\mathbb{R}^N\setminus R_\lambda^{-1}(B^\lambda_{\delta_\varepsilon}))$.   Then $\varphi:=\psi_\varepsilon f$ is a suitable test function for \eqref{wsol}, that is 
$$\int_{\R^N}|\nabla u|^{p-2}\langle \nabla u,\nabla \varphi\rangle \,dx=\int_{\R^N}c(x)u^q\varphi \,dx.$$
\end{lem}

\begin{lem}\label{Lemmadisommabilità}
Let $p>1$, and let $u$ and $u_\lambda$ be solutions of \eqref{eq:Prob} and 
\eqref{ulambdasoddisfa}, respectively, with $u$ satisfying assumption \eqref{ipotesiu}.  
We assume that $c(x)\in L^{\infty}(\mathbb{R}^N)$ is positive and non‑decreasing in the set 
\(\{x_1<0\}\), and that 
\(\Gamma\subset\{x_1=0\}\) is a compact set satisfying \(\operatorname{Cap}_p(\Gamma)=0\). 
Then,

\begin{equation*}
   \int_{\Sigma_\lambda} (|\nabla u|+|\nabla u_\lambda|)^{p-2}|\nabla w_\lambda^{+}|^{2}\,dx
\le
C(q,\|c\|_\infty)\|u\|_{L^{q+1}(\Sigma_\lambda)}^{q+1}, 
\end{equation*} 
where $C(q,\|c\|_\infty)$ is a positive constant.

\end{lem}

\begin{proof}
The proof is similar to that of \cite[Lemma~2.2]{EMS}, and for this reason we only 
outline the necessary modifications. We consider the following function
\begin{equation*}
 \varphi_\varepsilon := w_\lambda^{+}\, \psi_\varepsilon^{\,p}\, \chi_{\Sigma_\lambda},   
\end{equation*}
where \(\psi_\varepsilon\) and \(w_\lambda^{+}\) have been introduced in \eqref{definizionecutoff} and \eqref{definizionewlambda}, respectively. Here \(\chi_{\Sigma_\lambda}\) denotes the characteristic function of the set \(\Sigma_\lambda\). 
From Lemma~\ref{test}, together with Remark~\ref{ustainLinfinitolontanodallorigine}, we 
deduce that $\varphi_\varepsilon$ can be used as a test function in both \eqref{wsol} and 
\eqref{ulambdasoddisfa}.
Subtracting the two identities we deduce 
\begin{equation}\label{2.15}
    \begin{split}
&\int_{\Sigma_\lambda} (|\nabla u|^{p-2}\nabla u - |\nabla u_\lambda|^{p-2}\nabla u_\lambda,\nabla w_\lambda^{+})\,\psi_\varepsilon^{\,p}\,dx
\\&=- p\int_{\Sigma_\lambda} (|\nabla u|^{p-2}\nabla u - |\nabla u_\lambda|^{p-2}\nabla u_\lambda,\nabla\psi_\varepsilon)\,\psi_\varepsilon^{\,p-1} w_\lambda^{+}\,dx
+\int_{\Sigma_\lambda} (c(x)u^q-c_\lambda (x)u_\lambda^q)\, w_\lambda^{+}\,\psi_\varepsilon^{\,p}\,dx.
    \end{split}
\end{equation}

We recall that, for any \(\eta,\eta' \in \mathbb{R}^N\) with \(|\eta|+|\eta'|>0\), there exists a positive constant \(C_1\), depending only on \(p\), such that
\begin{equation}\label{Damascelli}
    \begin{split}
\big(|\eta|^{p-2}\eta - |\eta'|^{p-2}\eta'\big)\cdot(\eta-\eta')
\;\ge\;
C_1 (|\eta|+|\eta'|)^{p-2} |\eta-\eta'|^2.
\end{split}
\end{equation}

By combining \eqref{Damascelli} with the fact that \(c(x)\le c_\lambda(x)\) in \(\Sigma_\lambda\), we obtain
\begin{equation*}
    \begin{split}
&C_1\int_{\Sigma_\lambda} (|\nabla u|+|\nabla u_\lambda|)^{p-2}|\nabla w_\lambda^{+}|^2\psi_\varepsilon^{\,p}\,dx
\\&\leq\int_{\Sigma_\lambda} (|\nabla u|^{p-2}\nabla u - |\nabla u_\lambda|^{p-2}\nabla u_\lambda,\nabla w_\lambda^{+})\,\psi_\varepsilon^{\,p}\,dx
\\&\leq  p\int_{\Sigma_\lambda} \left||\nabla u|^{p-2}\nabla u - |\nabla u_\lambda|^{p-2}\nabla u_\lambda\right||\nabla\psi_\varepsilon|\,\psi_\varepsilon^{\,p-1} w_\lambda^{+}\,dx
+\int_{\Sigma_\lambda} c(x)(u^q-u_\lambda^q)\, w_\lambda^{+}\,\psi_\varepsilon^{\,p}\,dx.
    \end{split}
\end{equation*}
Now we estimate the second term on the right-hand side of \eqref{2.15}. 
By exploiting the fact that \(c\in L^{\infty}(\mathbb{R}^N)\), together with the inequalities 
\(u\ge u_\lambda>0\) and the convexity of the map \(t\mapsto t^{q}\) for \(t>0\), we obtain
\begin{equation}\label{stimetta}
    \begin{split}
&C_1\int_{\Sigma_\lambda} (|\nabla u|+|\nabla u_\lambda|)^{p-2}|\nabla w_\lambda^{+}|^2\psi_\varepsilon^{\,p}\,dx
\\&\leq  p\int_{\Sigma_\lambda} \left||\nabla u|^{p-2}\nabla u - |\nabla u_\lambda|^{p-2}\nabla u_\lambda\right||\nabla\psi_\varepsilon|\,\psi_\varepsilon^{\,p-1} w_\lambda^{+}\,dx
+q\|c\|_{L^{\infty}}\int_{\Sigma_\lambda} u^{q-1}\, (w_\lambda^{+})^2\,\psi_\varepsilon^{\,p}\,dx
\\&\leq  p\int_{\Sigma_\lambda} \left||\nabla u|^{p-2}\nabla u - |\nabla u_\lambda|^{p-2}\nabla u_\lambda\right||\nabla\psi_\varepsilon|\,\psi_\varepsilon^{\,p-1} w_\lambda^{+}\,dx
+q\|c\|_{L^{\infty}}\int_{\Sigma_\lambda} u^{q+1}\, \psi_\varepsilon^{\,p}\,dx, 
    \end{split}
\end{equation}
where in the last inequality we have used the fact that $w_{\lambda}^+\leq u$ in $\Sigma_{\lambda}$.
Repeating verbatim the decomposition in \cite[Lemma~2.2]{EMS} from (2.18) to (2.21), 
we obtain, in the case $p<2$, that
\begin{equation}\label{stimanelcasop<2}
    \begin{split}
&\int_{\Sigma_\lambda} (|\nabla u|+|\nabla u_\lambda|)^{p-2}|\nabla w_\lambda^{+}|^{2}\psi_\varepsilon^{\,p}\,dx
\\&\le
\delta C(p)\,\int_{\Sigma_\lambda} (|\nabla u|+|\nabla u_\lambda|)^{p-2}|\nabla w_\lambda^{+}|^{2}\psi_\varepsilon^{\,p}\,dx
+ C\int_{\Sigma_\lambda} |\nabla\psi_\varepsilon|^{p}\,dx\\&+C\left(\int_{\Sigma_\lambda} |\nabla u|^{p}\,dx\right)^{\frac{p-1}{p}}
\left(\int_{\Sigma_\lambda} |\nabla\psi_\varepsilon|^{p}\,dx\right)^{\frac{1}{p}}
+ q\|c\|_{L^{\infty}}\int_{\Sigma_\lambda} u^{q+1}\,dx,
\end{split}
\end{equation}
where $\delta>0$, $C(p)$ and $C=C(\delta,p,\lambda,\|u\|_{L^{\infty}(\Sigma_\lambda)})$ are positive constants.

In the case $p> 2$, repeating verbatim the decomposition in \cite[Lemma~2.2]{EMS}, 
from (2.22) to (2.25), we deduce that
\begin{equation}\label{stimanelcasopmaggioreugualea2}
\begin{split}
&\int_{\Sigma_\lambda} (|\nabla u|+|\nabla u_\lambda|)^{p-2}|\nabla w_\lambda^{+}|^{2}\psi_\varepsilon^{\,p}\,dx
\\&\le
\delta C(p)\int_{\Sigma_\lambda} (|\nabla u|+|\nabla u_\lambda|)^{p-2}|\nabla w_\lambda^{+}|^{2}\psi_\varepsilon^{\,p}\,dx
\\&+\,C\left(\int_{\Sigma_\lambda} |\nabla u|^{p}\,dx\right)^{\frac{p-2}{p}}
\left(\int_{\Sigma_\lambda} |\nabla\psi_\varepsilon|^{p}\,dx\right)^{\frac{2}{p}}
+ C\int_{\Sigma_\lambda} |\nabla\psi_\varepsilon|^{p}\,dx
+q\|c\|_{L^{\infty}}\int_{\Sigma_\lambda} u^{q+1}\,dx,
\end{split}
\end{equation}
where $\delta>0$, $C(p)$ and $C=C(\delta,p,\lambda,\|u\|_{L^{\infty}(\Sigma_\lambda)})$ are positive constants.

We now choose \(\delta>0\) sufficiently small so that the first term on the right-hand side of \eqref{stimanelcasop<2} and \eqref{stimanelcasopmaggioreugualea2} can be absorbed into the left-hand side. Using \eqref{Proprietacutoffdisingolareriflesso}, and recalling that for \(\lambda<0\) the solution satisfies \(u\in D^{1,p}(\Sigma_\lambda)\), see Remark \ref{ustainLinfinitolontanodallorigine}, we may let \(\varepsilon\to 0\). By Fatou’s Lemma, we obtain
\begin{equation*}
   \int_{\Sigma_\lambda} (|\nabla u|+|\nabla u_\lambda|)^{p-2}|\nabla w_\lambda^{+}|^{2}\,dx
\le
C(q,\|c\|_{L^{\infty}})\int_{\Sigma_\lambda} u^{q+1}\,dx, 
\end{equation*}
where $C(q,\|c\|_{L^{\infty}})$ is a positive constant. By Remark \ref{ustainLinfinitolontanodallorigine} we note that the r.h.s. of the latter inequality is bounded. This concludes the proof.
\end{proof}

We now state a small–domain principle comparing the solution $u$ with its reflection 
$u_\lambda$ in a neighborhood of small measure around the reflected singular set 
$R_\lambda(\Gamma)$.

\begin{lem}\label{dominipiccoli}
    Let \(p>2\), and let \(u\) and \(u_\lambda\) be solutions of \eqref{eq:Prob} and \eqref{ulambdasoddisfa}, respectively, with $u$ satisfying assumption \eqref{ipotesiu}. 
We assume that $c(x)\in L^{\infty}(\mathbb{R}^N)$ is positive and non‑decreasing in the 
half‑space \(\{x_1<0\}\). Moreover, we suppose that 
\(\Gamma\subset\{x_1=0\}\) is a compact set with $\operatorname{Cap}_p(\Gamma)=0$. There exists a constant 
\(\tau=\tau(n,q,\|c\|_{\infty},\|u\|_{L^{\infty}(\Sigma_\lambda)})>0\) 
such that, for every neighborhood \(I(R_\lambda(\Gamma))\) of \(R_\lambda(\Gamma)\) satisfying 
\(|I(R_\lambda(\Gamma))|<\tau\), the following holds:  
if \(u \le u_\lambda\) on $\partial I(R_\lambda(\Gamma)),$ then \[
u \le u_\lambda \quad \text{in } I(R_\lambda(\Gamma))\setminus R_\lambda(\Gamma).
\]
If \(1<p<2\), the same conclusion holds provided that \(\Gamma\) satisfies the assumption of 
Theorem~\ref{Teopminore2}.

\end{lem}
\begin{proof}
For $\varepsilon>0$ small enough, we consider the function
\begin{equation*}
 \varphi_\varepsilon := w_\lambda^{+}\, \psi_\varepsilon^{\,p}\, \chi_{I(R_{\lambda}(\Gamma))},
\end{equation*}
where \(\psi_\varepsilon\) and \(w_\lambda^{+}\) have been introduced in \eqref{definizionecutoff} and \eqref{definizionewlambda}, respectively. Since, by construction, \(w_\lambda^{+} \le \|u\|_{L^\infty(\Sigma_\lambda)}\) and $u\leq u_\lambda$ on $\partial I(R_{\lambda}(\Gamma))$, standard arguments ensure that \(\varphi_\varepsilon \in W^{1,p}_0(I(R_{\lambda}(\Gamma)))\). By a density argument, we may use \(\varphi_\varepsilon\) as a test function in \eqref{wsol} and in \eqref{ulambdasoddisfa}.

    \textbf{Case 1: \(p \geq 2\).}
    Proceeding as in the proof of Lemma \ref{Lemmadisommabilità} in the case \(p \ge 2\), we obtain the following estimate (see \eqref{stimanelcasopmaggioreugualea2}):
\begin{equation}\label{stimanelcasopmaggioreugualea2nellemmadopodopo}
\begin{split}
&\int_{I(R_\lambda(\Gamma))} (|\nabla u|+|\nabla u_\lambda|)^{p-2}|\nabla w_\lambda^{+}|^{2}\psi_\varepsilon^{\,p}\,dx
\\&\le
\delta C(p)
\int_{I(R_\lambda(\Gamma))} (|\nabla u|+|\nabla u_\lambda|)^{p-2}|\nabla w_\lambda^{+}|^{2}\psi_\varepsilon^{\,p}\,dx
\\&\quad
+\,C\left(\int_{I(R_\lambda(\Gamma))} |\nabla u|^{p}\,dx\right)^{\frac{p-2}{p}}
\left(\int_{I(R_\lambda(\Gamma))} |\nabla\psi_\varepsilon|^{p}\,dx\right)^{\frac{2}{p}}
+ C\int_{I(R_\lambda(\Gamma))} |\nabla\psi_\varepsilon|^{p}\,dx
\\&\quad
+ q\|c\|_{L^{\infty}}
\int_{I(R_\lambda(\Gamma))} u^{q-1}\, (w_\lambda^{+})^2\,\psi_\varepsilon^{\,p}\,dx.
\end{split}
\end{equation}

Choosing \(\delta>0\) sufficiently small, and using \eqref{Proprietacutoffdisingolareriflesso}, together with the fact that for \(\lambda<0\) the solution satisfies \(u\in D^{1,p}(I(R_\lambda(\Gamma)))\cap L^{\infty}(I(R_\lambda(\Gamma)))\), we may pass to the limit as \(\varepsilon\to 0\). By Fatou’s Lemma, we obtain
\begin{equation}\label{fazzi2}
\begin{split}
\int_{I(R_\lambda(\Gamma))} (|\nabla u|+|\nabla u_\lambda|)^{p-2}|\nabla w_\lambda^{+}|^{2}\,dx
\le
C(q,\|c\|_{L^{\infty}},\|u\|_{L^{\infty}(\Sigma_\lambda)})
\int_{I(R_\lambda(\Gamma))} \, (w_\lambda^{+})^2\,dx.
\end{split}
\end{equation}
Now we set \(\rho := |\nabla u|^{p-2}\), and observe that \(\rho\) is bounded in 
\(I(R_\lambda(\Gamma))\). Hence \(\rho \in L^{1}(I(R_\lambda(\Gamma)))\).
Applying the weighted Poincaré inequality to \eqref{fazzi2}, see \cite[Theorem 1.2]{DS1}, and since $|\nabla u|^{p-2} \le (|\nabla u| + |\nabla u_{\lambda}|)^{p-2}$, we obtain
\begin{equation}\label{3.47}
\begin{split}
    &\int_{I(R_\lambda(\Gamma))} 
\rho\, |\nabla w_\lambda^{+}|^{2}\,dx
\le C(q,\|c\|_{L^{\infty}},\|u\|_{L^{\infty}(\Sigma_\lambda)})\int_{I(R_\lambda(\Gamma))} 
(w_\lambda^{+})^{2}\,dx
\\&\le 
C(q,\|c\|_{L^{\infty}},\|u\|_{L^{\infty}(\Sigma_\lambda)})C_p(|I(R_\lambda(\Gamma))|)
\int_{I(R_\lambda(\Gamma))} 
\rho\, |\nabla w_\lambda^{+}|^{2}\,dx,
\end{split}
\end{equation}
where \(C_{p}(\cdot)\to 0\) as the measure of the domain tends to zero. For \(\tau>0\) sufficiently small and by \eqref{3.47}, we deduce
\[
\int_{I(R_\lambda(\Gamma))} 
\rho\, |\nabla w_\lambda^{+}|^{2}\,dx \le 0.
\]
This proves that $u \le u_{\lambda}$ in $I(R_\lambda(\Gamma))\setminus R_\lambda(\Gamma)$.

    \textbf{Case 2: \(1<p<2\).}
Proceeding as in the proof of Lemma \ref{Lemmadisommabilità} in the case \(1<p <2\), we obtain the following estimate (see \eqref{stimanelcasop<2}):

\begin{equation}\label{stimanelcasop<2dominipiccoli}
    \begin{split}
&\int_{I(R_\lambda(\Gamma))} (|\nabla u|+|\nabla u_\lambda|)^{p-2}|\nabla w_\lambda^{+}|^{2}\psi_\varepsilon^{\,p}\,dx
\\&\le
\delta C(p)\,\int_{I(R_\lambda(\Gamma))} (|\nabla u|+|\nabla u_\lambda|)^{p-2}|\nabla w_\lambda^{+}|^{2}\psi_\varepsilon^{\,p}\,dx
+ C\int_{I(R_\lambda(\Gamma))} |\nabla\psi_\varepsilon|^{p}\,dx\\&+C\left(\int_{I(R_\lambda(\Gamma))} |\nabla u|^{p}\,dx\right)^{\frac{p-1}{p}}
\left(\int_{I(R_\lambda(\Gamma))} |\nabla\psi_\varepsilon|^{p}\,dx\right)^{\frac{1}{p}}
+ q\|c\|_{L^{\infty}}\int_{I(R_\lambda(\Gamma))} u^{q-1}(w_\lambda^+)^2\,dx,
\end{split}
\end{equation}

Choosing \(\delta>0\) sufficiently small, by Fatou’s Lemma, we obtain
\begin{equation}\label{granfazzi2}
\begin{split}
\int_{I(R_\lambda(\Gamma))} (|\nabla u|+|\nabla u_\lambda|)^{p-2}|\nabla w_\lambda^{+}|^{2}\,dx
\le
C(q,\|c\|_{L^{\infty}},\|u\|_{L^{\infty}(\Sigma_\lambda)})
\int_{I(R_\lambda(\Gamma))} \, (w_\lambda^{+})^2\,dx.
\end{split}
\end{equation}
Now we want to apply  the weighted Sobolev inequality with the following weight
\[
\rho := \bigl(1 + |\nabla u|^2 + |\nabla u_{\lambda}|^2\bigr)^{\frac{p-2}{2}}.
\]
We note that $\rho \in L^1\bigl(I(R_{\lambda}(\Gamma))\bigr)$ and  by \cite[Lemma 3.1]{EMS}, we have 
\[
\rho^{-1}
= \bigl(1 + |\nabla u|^{2} + |\nabla u_{\lambda}|^{2}\bigr)^{\frac{2-p}{2}}
\in L^{t}\bigl(I(R_{\lambda}(\Gamma))\cap \operatorname{supp}(w_{\lambda}^{+})\bigr)
\quad\text{for some } t> \frac{n}{2},
\]
which allows us to apply the weighted Sobolev inequality (see \cite{DS1}).

We recall that the space \(H^{1}_{0,\rho}\bigl(I(R_{\lambda}(\Gamma))\bigr)\) (see \cite{DS1}) coincides with the closure of
\(C_c^\infty\bigl(I(R_{\lambda}(\Gamma))\bigr)\) with respect to the norm
\[
\|w\|_\rho := 
\left(\int_{I(R_{\lambda}(\Gamma))} \rho\,|\nabla w|^2\,dx\right)^{1/2},
\]
and the weighted Sobolev inequality yields
\[
\|w\|_{L^{2^*_\rho}(I(R_{\lambda}(\Gamma)))} 
\le C_p(|\Omega '|) \|\nabla w\|_{L^2(I(R_{\lambda}(\Gamma)),\rho)}
\quad\text{for all } w\in H^{1}_{0,\rho}\bigl(I(R_{\lambda}(\Gamma))\bigr),
\]
where $C_p(I(R_{\lambda}(\Gamma)))\rightarrow 0$ if $|I(R_{\lambda}(\Gamma))|\rightarrow 0$ and 
\[
\frac{1}{2^*_\rho}
= \frac{1}{2} - \frac{1}{n} + \frac{1}{2t}.
\]
We shall also use the elementary estimate
\begin{equation}\label{eq:grad-sum-lambda}
\bigl(|\nabla u| + |\nabla u_{\lambda}|\bigr)^{2-p}
\le 2^{\frac{2-p}{2}}
\bigl(|\nabla u|^{2} + |\nabla u_{\lambda}|^{2}\bigr)^{\frac{2-p}{2}}
\le
2^{\frac{2-p}{2}}
\bigl(1 + |\nabla u|^{2} + |\nabla u_{\lambda}|^{2}\bigr)^{\frac{2-p}{2}}.
\end{equation}
Using \eqref{eq:grad-sum-lambda}, Hölder's inequality, and the weighted Sobolev inequality, the inequality \eqref{granfazzi2} becomes 
\begin{equation}
\begin{split}
&\int_{I(R_{\lambda}(\Gamma))} 
\rho\,|\nabla w_{\lambda}^{+}|^{2}\,dx
\le
2^{\frac{2-p}{2}}
\int_{I(R_{\lambda}(\Gamma))}
\bigl(|\nabla u| + |\nabla u_{\lambda}|\bigr)^{p-2}
|\nabla w_{\lambda}^{+}|^{2}\,dx
\\&\le
2^{\frac{2-p}{2}}
C(q,\|c\|_{L^{\infty}},\|u\|_{L^{\infty}(\Sigma_\lambda)})
\left|I(R_{\lambda}(\Gamma))\right|^{\frac{2^*_\rho-2}{2^*_\rho}}
\left(
\int_{I(R_{\lambda}(\Gamma))}
(w_{\lambda}^{+})^{2^*_{\rho}}\,dx
\right)^{\frac{2}{2^*_{\rho}}}
\\&\le
2^{\frac{2-p}{2}}
C(q,\|c\|_{L^{\infty}},\|u\|_{L^{\infty}(\Sigma_\lambda)})\left|I(R_{\lambda}(\Gamma))\right|^{\frac{2^*_\rho-2}{2^*_\rho}}
C_p\bigl(|I(R_{\lambda}(\Gamma))|\bigr)
\int_{I(R_{\lambda}(\Gamma))}
\rho\,|\nabla w_{\lambda}^{+}|^{2}\,dx.
\end{split}
\end{equation}
Since the measure of the tubular neighborhood \(I(R_{\lambda}(\Gamma))\) can be chosen arbitrarily small, we may assume that
\[
2^{\frac{2-p}{2}}
C(q,\|c\|_{L^{\infty}},\|u\|_{L^{\infty}(\Sigma_\lambda)})\left|I(R_{\lambda}(\Gamma))\right|^{\frac{2^*_\rho-2}{2^*_\rho}}
C_p\bigl(|I(R_{\lambda}(\Gamma))|\bigr) < \frac{1}{2}.
\]
This implies
\[
\int_{I(R_{\lambda}(\Gamma))} 
\rho\,|\nabla w^+_{\lambda}|^2\,dx = 0,
\]
and therefore \(u \le u_{\lambda}\) in \(I(R_{\lambda}(\Gamma))\).  

\end{proof}

As a first step, we show that $u$ is monotone at infinity; more precisely, we prove that 
\(\Lambda\neq\emptyset\) for every $\lambda<0$ with $|\lambda|$ sufficiently large.
\begin{prop}\label{monotoniaallinfinito}
  Let \(p>2\), and let \(u\) and \(u_\lambda\) be solutions of \eqref{eq:Prob} and \eqref{ulambdasoddisfa}, respectively, with $u$ satisfying assumption \eqref{ipotesiu}. 
We assume that $c(x)\in L^{\infty}(\mathbb{R}^N)$ is positive and non‑decreasing in the 
half‑space \(\{x_1<0\}\). Moreover, we suppose that 
\(\Gamma\subset\{x_1=0\}\) is a compact set with $\operatorname{Cap}_p(\Gamma)=0$. There exists a value $\mu<0$ sufficiently negative such that
\[u \le u_\lambda, \quad \text{in } \Sigma_\lambda, \quad \forall \lambda\leq \mu.
\]
If \(1<p<2\), the same conclusion holds provided that \(\Gamma\) satisfies the assumption of 
Theorem~\ref{Teopminore2}. 
\end{prop}

\begin{proof}
Fix \(\hat\lambda<0\) sufficiently negative so that \(\Sigma_{\hat\lambda}\subset\subset \R^N\setminus B_{\hat R}\), where \(\hat R\) is given by Theorem \ref{thm:asymptEst} and Theorem \ref{stimagradiente}. Let \(\lambda<\hat\lambda\), and let \(\tau>0\), independent of \(\lambda\), be the constant given by Lemma \ref{dominipiccoli}. Let \(I(R_\lambda(\Gamma))\) be a neighborhood of \(R_\lambda(\Gamma)\) such that \(|I(R_\lambda(\Gamma))|<\tau\). Since \(|u(x)|\to 0\) as \(|x|\to+\infty\) by Theorem \ref{thm:asymptEst}, there exists \(\lambda'<\hat\lambda\) such that, for every \(\lambda<\lambda'\), one has \(u\le u_\lambda\) on \(\partial I(R_\lambda(\Gamma))\). By Lemma \ref{dominipiccoli}, this implies \(u\le u_\lambda\) in \(I(R_\lambda(\Gamma))\).

We now aim to prove the existence of \(\tilde\lambda<\lambda'\), sufficiently negative, such that for every \(\lambda<\tilde \lambda\) one has \(u\le u_\lambda\) in \(\Sigma_\lambda\setminus I(R_\lambda(\Gamma))\). To this end, we split the proof into two cases.

\textbf{Case 1: \(p \geq 2\).}

We consider the following function
\begin{equation*}
 \varphi := w_\lambda^{+}\, \chi_{\Sigma_\lambda},   \end{equation*}
where \(w_\lambda^{+}\) has been introduced in  \eqref{definizionewlambda}.  Since $u\leq u_\lambda$ in $\partial \Sigma_\lambda$,  by density argument, see Remark \ref{ustainLinfinitolontanodallorigine} and Lemma \ref{test}, we may use \(\varphi\) as a test function in \eqref{wsol} and in \eqref{ulambdasoddisfa}. Subtracting the two identities and proceeding as in the proof of Lemma \ref{Lemmadisommabilità}, we obtain the following inequality (see also \eqref{stimetta}), that is 
\begin{equation}\label{fazzi}
\begin{split}
C_1\int_{\Sigma_\lambda} (|\nabla u|+|\nabla u_\lambda|)^{p-2}|\nabla w_\lambda^{+}|^2\,dx
\leq q\|c\|_{L^{\infty}}\int_{\Sigma_\lambda} u^{q-1}\, (w_\lambda^{+})^2\,\,dx.
  \end{split}
\end{equation}

Now we estimate the right-hand side of \eqref{fazzi}. Using Theorem \ref{thm:asymptEst}, we obtain
\begin{equation}\label{eq:batiado}
\begin{split}
\int_{\Sigma_\lambda} u^{q-1}\, (w_\lambda^{+})^2\,dx\le \hat C\int_{\Sigma_\lambda}
\frac{(w_\lambda^{+})^{2}}{|x|^{(q-1)\frac{N-p}{p-1}}}\,dx.
\end{split}
\end{equation}
At this stage, we intend to apply Hardy’s inequality (see \cite[Lemma 2.3]{DR}), and therefore we introduce
\begin{equation}\label{definizionedisebetastar}
s := -\left(\frac{N-p}{p-1}+1\right)(p-2),
\qquad
\beta^{*} := (q-1)\frac{N-p}{p-1} + s - 2.
\end{equation}
It is easy to check that $\beta^{*}>0$, for $q>N(p-1)/(N-p)$, and $s>2-N$, so Hardy’s inequality applies.  
By \eqref{fazzi}, \eqref{eq:batiado}, and noting that $|x|>|\lambda|$ in $\Sigma_\lambda$, we obtain
\begin{equation}\label{eq:batiadohfhfhfh}
\begin{split}
&\int_{\Sigma_\lambda} (|\nabla u| + |\nabla u_\lambda|)^{p-2}
|\nabla  w_\lambda^{+}|^{2}\,dx
\\&\le
C(q,\|c\|_{L^\infty})\hat C\int_{\Sigma_\lambda}
\frac{1}{|x|^{(q-1)\frac{N-p}{p-1}}}(w_\lambda^{+})^{2}\,dx\\&
=C(q,\|c\|_{L^\infty})\hat C\int_{\Sigma_\lambda}
\frac{|x|^{s-2}}{|x|^{\beta^*}}(w_\lambda^{+})^{2}\,dx
\\&\le
\frac{C(q,\|c\|_{L^\infty})\hat C}{|\lambda|^{\beta^{*}}}
\left(\frac{2}{N+s-2}\right)^{2}
\int_{\Sigma_\lambda} |x|^{s}|\nabla w_\lambda^{+}|^{2}\,dx \\
&\le
\frac{C(q,\|c\|_{L^\infty})\hat C}{|\lambda|^{\beta^{*}}}
\left(\frac{2}{N+s-2}\right)^{2}
\tilde c^{2-p}\int_{\Sigma_\lambda} |\nabla u|^{p-2}|\nabla w_\lambda^{+}|^{2}\,dx \\
&\le
\frac{C(q,\|c\|_{L^\infty})\hat C}{|\lambda|^{\beta^{*}}}
\left(\frac{2}{N+s-2}\right)^{2}
\tilde c^{2-p}\int_{\Sigma_\lambda} (|\nabla u| + |\nabla u_\lambda|)^{p-2}
|\nabla w_\lambda^{+}|^{2}\,dx,
\end{split}
\end{equation}
where we have used Theorem \ref{stimagradiente}. For $\lambda<\tilde \lambda$ sufficiently negative such that
\[
\frac{C(q,\|c\|_{L^\infty})\hat C}{|\lambda|^{\beta^{*}}}
\left(\frac{2}{N+s-2}\right)^{2}
\tilde c^{2-p} < 1,
\]
inequality \eqref{eq:batiadohfhfhfh} yields a contradiction unless $(u - u_\lambda)^{+}=0$.  
Thus,
\[
u \le u_\lambda \qquad \text{in } \Sigma_\lambda\setminus R_\lambda(\Gamma).
\]

\textbf{Case 2: \(1<p<2\).}
We recall from the beginning of the proof that there exists a value $\lambda' < 0$, sufficiently negative, such that 
$u \leq u_{\lambda}$ in $I(R_{\lambda}(\Gamma))$ for every $\lambda < \lambda'$. 
As before, we consider the ball $B_{\hat R}$, where $\hat R$ is the radius given by 
Theorem~\ref{thm:asymptEst} and Theorem~\ref{stimagradiente}. 
By Theorem~\ref{thm:asymptEst}, we also know that $u(x) \to 0$ as $|x| \to +\infty$. 
This implies the existence of a value $\bar\lambda < \lambda'$ such that 
$u \leq u_{\lambda}$ in $R_\lambda(B_{\hat R}) \setminus I(R_{\lambda}(\Gamma))$ for every $\lambda < \bar\lambda$. It remains to prove that $u \leq u_{\lambda}$ in $\Sigma_{\lambda} \setminus R_\lambda(B_{\hat R}) $.
We consider the following function
\begin{equation*}
 \varphi := w_\lambda^{+}\, \chi_{\Sigma_\lambda\setminus R_\lambda(B_{\hat R})} .   \end{equation*}
  Since $u\leq u_\lambda$ in $\partial (\Sigma_\lambda\setminus R_\lambda(B_{\hat R}) )$,  by density argument, see Remark \ref{ustainLinfinitolontanodallorigine} and by Lemma \ref{test}, we may use \(\varphi\) as a test function in \eqref{wsol} and in \eqref{ulambdasoddisfa}. Subtracting the two identities and proceeding as in the proof of Lemma \ref{Lemmadisommabilità}, we obtain the following inequality (see also \eqref{stimetta}), that is
\begin{equation*}
\begin{split}
C_1\int_{\Sigma_\lambda\setminus R_\lambda(B_{\hat R}) } (|\nabla u|+|\nabla u_\lambda|)^{p-2}|\nabla w_\lambda^{+}|^2\,dx
\leq q\|c\|_{L^{\infty}}\int_{\Sigma_\lambda\setminus R_\lambda(B_{\hat R}) } u^{q-1}\, (w_\lambda^{+})^2\,dx.
  \end{split}
\end{equation*}
As in the previous case, $p\geq 2$, we deduce (see \eqref{eq:batiadohfhfhfh}):
\begin{equation}\label{eq:batiadohfhfhfh454}
\begin{split}
&\int_{\Sigma_\lambda\setminus R_\lambda(B_{\hat R}) } (|\nabla u| + |\nabla u_\lambda|)^{p-2}
|\nabla  w_\lambda^{+}|^{2}\,dx
\\&\le
\frac{C(q,\|c\|_{L^\infty})\hat C}{|\lambda|^{\beta^{*}}}
\left(\frac{2}{N+s-2}\right)^{2}
\int_{\Sigma_\lambda\setminus R_\lambda(B_{\hat R}) } |x|^{s}|\nabla w_\lambda^{+}|^{2}\,dx 
\end{split}
\end{equation}
where $s$ and $\beta ^*$ are defined as in \eqref{definizionedisebetastar}.
We note that, since $u\le u_\lambda$ in the support of $(u-u_\lambda)^+$, by 
Theorem~\ref{thm:asymptEst}, we deduce
\[
\frac{\hat C}{|x|^{\frac{N-p}{p-1}}}
\;\ge\;
u(x)
\;\ge\;
u_\lambda(x)
\;\ge\;
\frac{\hat c}{|x_\lambda|^{\frac{N-p}{p-1}}}
\qquad\text{in }\Sigma_\lambda\setminus R_\lambda(B_{\hat R}).
\]
This implies \(
|x|\le
\left(\frac{\hat C}{\hat c}\right)^{\frac{p-1}{N-p}}
|x_\lambda|.\) Hence by Theorem \ref{stimagradiente}, we have 
\begin{equation}\label{cose}
    (|\nabla u|+|\nabla u_\lambda|)^{2-p}\leq \tilde C^{2-p}\left(\frac{1}{|x|^{\frac{N-1}{p-1}}}+\frac{1}{|x_\lambda|^{\frac{N-1}{p-1}}}\right)^{2-p}\leq \frac{\bar C}{|x|^{\frac{N-1}{p-1}(2-p)}}\quad \text{in }\Sigma_\lambda\setminus R_\lambda(B_{\hat R}) ,
\end{equation}
where $\bar C$ is a positive constant depending on $N$ and $p$. By \eqref{eq:batiadohfhfhfh454} and \eqref{cose}, it follows that 
\begin{equation}\label{bellosono}
    \begin{split}
&\int_{\Sigma_\lambda\setminus R_\lambda(B_{\hat R}) } (|\nabla u| + |\nabla u_\lambda|)^{p-2}
|\nabla  w_\lambda^{+}|^{2}\,dx
\\&\le
\frac{C(q,\|c\|_{L^\infty})\hat C}{|\lambda|^{\beta^{*}}}
\left(\frac{2}{N+s-2}\right)^{2}
\int_{\Sigma_\lambda\setminus R_\lambda(B_{\hat R}) } |x|^{s}|\nabla w_\lambda^{+}|^{2}\,dx
\\&=\frac{C(q,\|c\|_{L^\infty})\hat C}{|\lambda|^{\beta^{*}}}
\left(\frac{2}{N+s-2}\right)^{2}
\int_{\Sigma_\lambda\setminus R_\lambda(B_{\hat R}) } |x|^{s}(|\nabla u| + |\nabla u_\lambda|)^{2-p}(|\nabla u| + |\nabla u_\lambda|)^{p-2}|\nabla w_\lambda^{+}|^{2}\,dx
\\&\leq \frac{C(q,\|c\|_{L^\infty})\hat C}{|\lambda|^{\beta^{*}}}
\left(\frac{2}{N+s-2}\right)^{2}\bar C
\int_{\Sigma_\lambda\setminus R_\lambda(B_{\hat R}) } (|\nabla u| + |\nabla u_\lambda|)^{p-2}|\nabla w_\lambda^{+}|^{2}\,dx.
    \end{split}
  \end{equation}
For $\lambda<\bar \lambda$ sufficiently negative such that
\[
\frac{C(q,\|c\|_{L^\infty})\hat C}{|\lambda|^{\beta^{*}}}
\left(\frac{2}{N+s-2}\right)^{2}\bar C < 1,
\]
inequality \eqref{bellosono} yields a contradiction unless $(u - u_\lambda)^{+}=0$.  
Thus,
\(u \le u_\lambda\) in \(\Sigma_{\lambda} \setminus R_\lambda(B_{\hat R}) \).
\end{proof}

We are now in a position to prove Theorem~\ref{Teopmaggiore2}.
\begin{proof}[Proof of Theorem \ref{Teopmaggiore2}]
   By Proposition~\ref{monotoniaallinfinito}, we obtain the existence of a value $\mu<0$ sufficiently negative such that for every $\lambda \leq \mu$, we have that  $u \leq u_{\lambda}$ in $\Sigma_{\lambda} \setminus R_{\lambda}(\Gamma)$. 
This fact shows that $\Lambda \neq \emptyset$ and hence the value $\mu_0 \leq 0$ is well defined, as in \eqref{lambdazero}.  
Moreover, by continuity we have $u \le u_{\mu_0}$ in $\Sigma_{\mu_0}\setminus R_{\mu_0}(\Gamma)$. 
We aim to prove that $\mu_0=0$. If $\mu_0<0$, then by the strong comparison 
principle \cite[Theorem~1.4]{DS2} we deduce that either $u<u_{\mu_0}$ in 
$\Sigma_{\mu_0}\setminus R_{\mu_0}(\Gamma)$ or $u=u_{\mu_0}$ in 
$\Sigma_{\mu_0}\setminus R_{\mu_0}(\Gamma)$; the equality is impossible due to the presence of the singular set $\Gamma$, and 
therefore $u<u_{\mu_0}$ in $\Sigma_{\mu_0}\setminus R_{\mu_0}(\Gamma)$.

For $\varepsilon,\delta >0$ and $R>\hat R$, that later on will be fixed large, and let us set 
\begin{equation}\label{rettangolo}
 \mathcal{P}^{\varepsilon,\delta}_R
:= \{(x_1,x_2,\dots,x_N) \in \mathbb{R}^N :
\mu_0 - \delta \le x_1 \le \mu_0 + \varepsilon,\;
|x_i| \le R \text{ for } i = 2,\dots,N \}.   
\end{equation}
Let us consider now $\bar N=\bar N(R)$ open cubes $Q_i^{\varepsilon,\delta}$ with edge $(\delta+\varepsilon)$ and with the $x_1$-coordinate of the center, say $x_{1,C}$, such that $x_{1,C}=\mu_0+(\varepsilon+\delta)/2$. Moreover we assume that $Q_i^{\varepsilon,\delta}\cap Q_j^{\varepsilon,\delta}=\emptyset$ for $i\neq j$ and 
\begin{equation}\label{unionediset}
\mathcal{P}^{\varepsilon,\delta}_R\subset \bigcup_{i=1}^{\bar N} \overline{{Q}_i^{\varepsilon,\delta}}=:Q_{R}^{\varepsilon,\delta}.
\end{equation}
Moreover, we set
\begin{equation}\label{definizionedisets}
    K_\delta:=\overline{B_R\cap \Sigma_{\mu_0-\delta}},\quad 
B_R^\varepsilon:=({Q}^{\varepsilon,\delta}_R\cup K_\delta)^c\cap \Sigma_{\mu_0+\varepsilon}.
\end{equation}

We note that $\Sigma_{\mu_0+\varepsilon}$ can be written as the disjoint union
$$\Sigma_{\mu_0+\varepsilon}=K_\delta\cup B_R^\varepsilon \cup Q_R^{\varepsilon,\delta}.$$
We want to prove that there exists $\bar \varepsilon>0$ such that $u\leq u_{\mu_0+\varepsilon}$, in $\Sigma_{\mu_0+\varepsilon}\setminus R_{\mu_0+\varepsilon}(\Gamma)$, for every $0<\varepsilon<\bar \varepsilon$. First of all, let $\tau>0$ be the constant given in Lemma~\ref{dominipiccoli}, and let 
$I_{\mu_0}$ be a neighbourhood of $R_{\mu_0}(\Gamma)$ such that $|I_{\mu_0}|<\tau$. 
Since $\Gamma\subset \{x_1=0\}$ is compact, there exist $\hat\varepsilon>0$ and $\hat\delta>0$ such that 
\[
R_{\mu_0+\varepsilon}(\Gamma)\subset\subset I_{\mu_0}\subset\subset K_\delta,\]
for every $0<\varepsilon<\hat\varepsilon$ and $0<\delta<\hat\delta$. Since $u < u_{\mu_0}$ in $\Sigma_{\mu_0}\setminus R_{\mu_0}(\Gamma)$, by uniform continuity  on $\partial I_{\mu_0}$, there exists $0<\varepsilon'<\hat \varepsilon$ such that $u<u_{\mu_0+\varepsilon}$ on $\partial I_{\mu_0}$, for every $0<\varepsilon<\varepsilon'$. By Lemma \ref{dominipiccoli}, we get 
\begin{equation}\label{piccolezzanellintorno}
    u<u_{\mu_0+\varepsilon},\quad \text{in } I_{\mu_0}\setminus R_{\mu_0+\varepsilon}(\Gamma),
\end{equation}
for every $0<\varepsilon<\varepsilon'$. Again, by uniform continuity there exist $\varepsilon^{''}=\varepsilon^{''}(\delta)<\varepsilon '$ such that 
\begin{equation}\label{piccolezzanelcompatto}
    u<u_{\mu_0+\varepsilon},\quad \text{in } K_{\delta}\setminus I_{\mu_0},
\end{equation}
for every $0<\varepsilon<\varepsilon^{''}$. 
Now we consider the following function $$\varphi=w_{\mu_0+\varepsilon}^+\chi_{\Sigma_{\mu_0+\varepsilon}},$$
where $w_{\mu_0+\varepsilon}^+:=(u-u_{\mu_0+\varepsilon})^+$. Using $\varphi$ as test function in \eqref{wsol} and in \eqref{ulambdasoddisfa}, subtracting the two identities and proceeding as in the proof of Lemma \ref{Lemmadisommabilità}, we obtain the following inequality (see also \eqref{stimetta}), that is 
\begin{equation}\label{fazzi400}
\begin{split}
&C_1\int_{B_R^\varepsilon\cup Q_R^{\varepsilon,\delta}} (|\nabla u|+|\nabla u_{\mu_0+\varepsilon}|)^{p-2}|\nabla w_{\mu_0+\varepsilon}^{+}|^2\,dx
=C_1\int_{\Sigma_{\mu_0+\varepsilon}} (|\nabla u|+|\nabla u_{\mu_0+\varepsilon}|)^{p-2}|\nabla w_{\mu_0+\varepsilon}^{+}|^2\,dx
\\&\leq q\|c\|_{L^{\infty}}\int_{B_R^\varepsilon\cup Q_R^{\varepsilon,\delta}} u^{q-1}\, (w_{\mu_0+\varepsilon}^{+})^2\,dx
\\&=q\|c\|_{L^{\infty}}\int_{B_R^\varepsilon} u^{q-1}\, (w_{\mu_0+\varepsilon}^{+})^2\,dx+q\|c\|_{L^{\infty}}\int_{ Q_R^{\varepsilon,\delta}} u^{q-1}\, (w_{\mu_0+\varepsilon}^{+})^2\,dx=:I_1+I_2.
  \end{split}
\end{equation}
We start estimating $I_1$. This term can be handled exactly as in the proof of 
Proposition~\ref{monotoniaallinfinito} in the case $p\ge 2$. 
Applying Hardy's inequality (see \cite[Proposition 1.1]{Sqaussina}) and using the fact that 
$|x|>R$, we obtain (see \eqref{eq:batiadohfhfhfh}):
\begin{equation}\label{termineBR}
I_1\leq \frac{C(q,\|c\|_{L^\infty})\hat c}{R^{\beta^{*}}}
\left(\frac{2}{N+s-2}\right)^{2}\tilde c^{2-p}
\int_{B_R^\varepsilon} (|\nabla u| + |\nabla u_{\mu_0+\varepsilon}|)^{p-2}|\nabla w_{\mu_0+\varepsilon}^{+}|^{2}\,dx.    
\end{equation}
First we fix $R>0$ sufficiently large such that
\begin{equation}\label{sceltadiR}
  \frac{C(q,\|c\|_{L^\infty})\hat c}{R^{\beta^{*}}}
\left(\frac{2}{N+s-2}\right)^{2}
\tilde c^{2-p} < 1.  
\end{equation}
Fo the term $I_2$, we proceed in the following way. The idea is considering the union \eqref{unionediset}, in order to use in each cube $Q_i^{\varepsilon,\delta}$ the weighted Poincaré inequality, see \cite[Corollary 5.3]{FMS3}. Thus, we deduce 
\begin{equation}\label{terminicubi}
\begin{split}
    I_2&\leq C(q,\|c\|_{\infty},\|u\|_{L^{\infty}(\mathcal P^{\varepsilon,\delta}_R)})\sum_{i=1}^{\bar N}\int_{Q_i^{\varepsilon,\delta}}  (w_{\mu_0+\varepsilon}^{+})^2\,dx
\\&\leq C(q,\|c\|_{\infty},\|u\|_{L^{\infty}(\mathcal P^{\varepsilon,\delta}_R)})\sum_{i=1}^{\bar N}C_p(Q_i^{\varepsilon,\delta})\int_{Q_i^{\varepsilon,\delta}}  |\nabla u|^{p-2}(w_{\mu_0+\varepsilon}^{+})^2\,dx
    \\& \leq C(q,\|c\|_{\infty},\|u\|_{L^{\infty}(\mathcal P^{\varepsilon,\delta}_R)})\sum_{i=1}^{\bar N}C_p(Q_i^{\varepsilon,\delta})\int_{Q_i^{\varepsilon,\delta}}  (|\nabla u|+|\nabla u_{\mu_0+\varepsilon}|)^{p-2}|\nabla w_{\mu_0+\varepsilon}^{+}|^2\,dx,
    \end{split}
\end{equation}
where in the last inequality we have used the fact $p\geq 2$, and $C_p(Q_i^{\varepsilon,\delta})\rightarrow 0$, if $|Q_i^{\varepsilon,\delta}|\rightarrow 0$.
Now we choose $\bar\delta>0$ such that 
\begin{equation}\label{sceltadidelta}
C(q,\|c\|_{\infty},\|u\|_{L^{\infty}(\mathcal P^{\varepsilon,\delta}_R)})\sum_{i=1}^{\bar N}C_p(Q_i^{0,\delta})<1/2, 
\end{equation}
for every $0<\delta<\bar\delta$.
Finally we choose $\bar\varepsilon<\varepsilon^{''}$, such that 
\begin{equation}\label{sceltadiepsilon}
C(q,\|c\|_{\infty},\|u\|_{L^{\infty}(\mathcal P^{\varepsilon,\delta}_R)})\sum_{i=1}^{\bar N}C_p(Q_i^{\varepsilon,\delta})<1, 
\end{equation}
for every $\varepsilon<\bar\varepsilon$. By the previous inequalities (from \eqref{piccolezzanellintorno} to \eqref{sceltadiepsilon}), we get 
$$\int_{\Sigma_{\mu_0+\varepsilon}} (|\nabla u|+|\nabla u_{\mu_0+\varepsilon}|)^{p-2}|\nabla w_{\mu_0+\varepsilon}^{+}|^2\,dx\leq 0.$$ Therefore, we get $u\leq u_{\mu_0+\varepsilon}$, in $\Sigma_{\mu_0+\varepsilon}\setminus R_{\mu_0+\varepsilon}(\Gamma)$, for every $0<\varepsilon<\bar \varepsilon$, that is a contradiction with the definition of $\mu_0$. Therefore $\mu_0=0$. 

Since the moving plane procedure can be performed in the same way but in the opposite direction, then this proves the desired symmetry result. The fact that the solution is increasing in the $x_1$-direction in $\{x_1<0\}$ is implicit in the moving plane procedure.
\end{proof}

We are now ready to proceed with the proof of Theorem~\ref{Teopminore2}.

\begin{proof}[Proof of Theorem \ref{Teopminore2}]
Arguing as in the proof of Theorem~\ref{Teopmaggiore2}, we obtain that $u \le u_{\mu_0}$ in $\Sigma_{\mu_0}\setminus R_{\mu_0}(\Gamma)$, where $\mu_0$ is defined as in \eqref{lambdazero}. 
We aim to prove that $\mu_0=0$. If $\mu_0<0$, by the strong comparison 
principle \cite[Theorem~2.5.2]{pucser} we deduce that, in any connected component 
$C$ of $\Sigma_{\mu_0}\setminus (Z_u^{\mu_0}\cup R_{\mu_0}(\Gamma))$, either 
$u=u_{\mu_0}$ or $u<u_{\mu_0}$, where 
$Z_u^{\mu_0}:=\{|\nabla u|+|\nabla u_{\mu_0}|=0\}$. Indeed, we note that the Lebesgue measure of $Z_u^{\mu_0}$ is zero 
(see \cite{DS1}) and, by Theorem~\ref{stimagradiente}, we have 
$Z_u^{\mu_0}\subset\subset B_{\hat R}$. 

If $\Sigma_{\mu_0}\setminus (Z_u^{\mu_0}\cup R_{\mu_0}(\Gamma))$ is connected, then the 
case $u=u_{\mu_0}$ on this set is impossible due to the presence of the critical set 
$\Gamma$. If instead 
$\Sigma_{\mu_0}\setminus (Z_u^{\mu_0}\cup R_{\mu_0}(\Gamma))$ has at least two connected 
components, since $Z_u\subset\subset B_{\hat R}$,  then exactly one of them is unbounded. In any other bounded connected component, the 
identity $u=u_{\mu_0}$ cannot occur in view of \cite[Lemma~3.2]{EMS}. Hence, we 
conclude that 
\[
u<u_{\mu_0}\quad\text{in }\Sigma_{\mu_0}\setminus (Z_u^{\mu_0}\cup R_{\mu_0}(\Gamma)).
\]

For $\varepsilon,\delta >0$ and $R>\hat R$, that later on will be fixed large, we recall that
$$\Sigma_{\mu_0+\varepsilon}=K_\delta\cup B_R^\varepsilon \cup \mathcal{P}^{\varepsilon,\delta}_R,$$
where $K_\delta$ and $\mathcal{P}^{\varepsilon,\delta}_R$ are defined in \eqref{rettangolo} and \eqref{definizionedisets}, and $B_R^\varepsilon:=(\mathcal{P}^{\varepsilon,\delta}_R\cup K_\delta)^c\cap \Sigma_{\mu_0+\varepsilon}$
First of all, let $\tau>0$ be the constant given in Lemma~\ref{dominipiccoli}, and let 
$I_{\mu_0}$ be a neighbourhood of $R_{\mu_0}(\Gamma)$ such that $|I_{\mu_0}|<\tau$. As in the proof of Theorem \ref{Teopmaggiore2}, there exist $\hat\varepsilon>0$ and $\hat\delta>0$ such that 
\[
R_{\mu_0+\varepsilon}(\Gamma)\subset\subset I_{\mu_0}\subset\subset K_\delta,\]
and 
\begin{equation}\label{piccolezzanellintornoparte2}
    u<u_{\mu_0+\varepsilon},\quad \text{in } I_{\mu_0},
\end{equation}
for every $0<\varepsilon<\hat\varepsilon$ and $0<\delta<\hat\delta$.

Now, for $l>0$ sufficiently small, we let $Z^{l,\mu_0}_u$ be an open set containing $Z_u^{\mu_0}$ such that $|Z^{l,\mu_0}_u|<l$. Again, by uniform continuity there exist $\varepsilon^{''}=\varepsilon^{''}(\delta)<\hat \varepsilon $ such that 
\begin{equation}\label{piccolezzanelcompattoparte2}
    u<u_{\mu_0+\varepsilon},\quad \text{in } K_{\delta}\setminus (I_{\mu_0}\cup Z^{l,\mu_0}_u),
\end{equation}
for every $0<\varepsilon<\varepsilon^{''}$. Moreover, by small comparison principle in small domains, see \cite{DS1}, we get $u\leq u_{\mu_0+\varepsilon}$ in $K_\delta\cap Z^{l,\mu_0}_u$. By the previous latter we get 
\begin{equation}\label{piccolezzanelcompattoparte22}
    u<u_{\mu_0+\varepsilon},\quad \text{in } K_{\delta},
\end{equation}
for every $0<\varepsilon<\varepsilon^{''}$.
Now we consider the following function $$\varphi=w_{\mu_0+\varepsilon}^+\chi_{\Sigma_{\mu_0+\varepsilon}},$$
where $w_{\mu_0+\varepsilon}^+:=(u-u_{\mu_0+\varepsilon})^+$. Using $\varphi$ as test function in \eqref{wsol} and in \eqref{ulambdasoddisfa}, subtracting the two identities and proceeding as in the proof of Lemma \ref{Lemmadisommabilità}, we obtain the following inequality (see also \eqref{stimetta}), that is 
\begin{equation}\label{fazzi4004}
\begin{split}
&C_1\int_{B_R^\varepsilon\cup Q_R^{\varepsilon,\delta}} (|\nabla u|+|\nabla u_{\mu_0+\varepsilon}|)^{p-2}|\nabla w_{\mu_0+\varepsilon}^{+}|^2\,dx
=C_1\int_{\Sigma_{\mu_0+\varepsilon}} (|\nabla u|+|\nabla u_{\mu_0+\varepsilon}|)^{p-2}|\nabla w_{\mu_0+\varepsilon}^{+}|^2\,dx
\\&\leq q\|c\|_{L^{\infty}}\int_{B_R^\varepsilon\cup \mathcal{P}^{\varepsilon,\delta}_R} u^{q-1}\, (w_{\mu_0+\varepsilon}^{+})^2\,dx
\\&=q\|c\|_{L^{\infty}}\int_{B_R^\varepsilon} u^{q-1}\, (w_{\mu_0+\varepsilon}^{+})^2\,dx+q\|c\|_{L^{\infty}}\int_{ \mathcal{P}^{\varepsilon,\delta}_R} u^{q-1}\, (w_{\mu_0+\varepsilon}^{+})^2\,dx=:I_1+I_2.
  \end{split}
\end{equation}
We estimate the term $I_1$. This term can be treated in the same way as in the proof of 
Proposition~\ref{monotoniaallinfinito} for the case $1<p<2$. 
Using Hardy's inequality (see \cite[Proposition 1.1]{Sqaussina}) and the fact that 
$|x|>R$, we deduce (see \eqref{bellosono}):
\begin{equation}\label{zaino}
   I_1\leq \frac{C(q,\|c\|_{L^\infty})\hat C}{|R|^{\beta^{*}}}
\left(\frac{2}{N+s-2}\right)^{2}\bar C
\int_{B_R^\varepsilon} (|\nabla u| + |\nabla u_{\mu_0+\varepsilon}|)^{p-2}|\nabla w_{\mu_0+\varepsilon}^{+}|^{2}\,dx. 
\end{equation}

First we fix $R$ large so that 
\begin{equation}\label{sceltaR2}
\frac{C(q,\|c\|_{L^\infty})\hat C}{|R|^{\beta^{*}}}
\left(\frac{2}{N+s-2}\right)^{2}\bar C < 1.   
\end{equation}
For the term $I_2$, using a one-dimensional-Poincar\'e inequality in the set $[\mu_0-\delta,\mu_0+\varepsilon]$, we have 
\begin{equation}\label{zaino97}
    \begin{split}
        I_2&\leq C(q,\|c\|_{L^{\infty}},\|u\|_{L^{\infty}(\mathcal{P}^{\varepsilon,\delta}_R)})\int_{ \mathcal{P}^{\varepsilon,\delta}_R} \, (w_{\mu_0+\varepsilon}^{+})^2\,dx
        \\&\leq C(q,\|c\|_{L^{\infty}},\|u\|_{L^{\infty}(\mathcal{P}^{\varepsilon,\delta}_R)})C_p(\delta,\varepsilon) \int_{ \mathcal{P}^{\varepsilon,\delta}_R} \, |\partial_{x_1} w_{\mu_0+\varepsilon}^{+}|^2
        \\&\leq  C_p(\delta,\varepsilon)\hat C\int_{ \mathcal{P}^{\varepsilon,\delta}_R} \, (|\nabla u|+|\nabla u_{\mu_0+\varepsilon}|)^{p-2}|\nabla w_{\mu_0+\varepsilon}^{+}|^2,
    \end{split}
\end{equation}
where $\hat C=\hat C(q,\|c\|_{L^{\infty}},\|u\|_{L^{\infty}(\mathcal{P}^{\varepsilon,\delta}_R)},\|\nabla u\|_{L^{\infty}(\mathcal{P}^{\varepsilon,\delta}_R)},\|\nabla u_{\mu_0+\varepsilon}\|_{L^{\infty}(\mathcal{P}^{\varepsilon,\delta}_R)})$ is a positive constant.
Now we choose $\bar\delta>0$ such that 
\begin{equation}\label{sceltadidelta2}
C_p(\delta,0)\hat C<1/2, 
\end{equation}
for every $0<\delta<\bar\delta$.
Finally we choose $\bar\varepsilon<\varepsilon^{''}$, such that 
\begin{equation}\label{sceltadiepsilon2}
C_p(\delta,\varepsilon)\hat C<1, 
\end{equation}
for every $\varepsilon<\bar\varepsilon$.
 By the previous inequalities (from \eqref{piccolezzanelcompattoparte22} to \eqref{sceltadiepsilon2}), we get 
$$\int_{\Sigma_{\mu_0+\varepsilon}} (|\nabla u|+|\nabla u_{\mu_0+\varepsilon}|)^{p-2}|\nabla w_{\mu_0+\varepsilon}^{+}|^2\,dx\leq 0.$$ Therefore, we get $u\leq u_{\mu_0+\varepsilon}$, in $\Sigma_{\mu_0+\varepsilon}\setminus R_{\mu_0+\varepsilon}(\Gamma)$, for every $0<\varepsilon<\bar \varepsilon$, that is a contradiction with the definition of $\mu_0$. Therefore, $\mu_0=0$.

Since the moving plane procedure can be performed in the same way but in the opposite direction, then this
proves the desired symmetry result. The fact that the solution is increasing in the $x_1$-direction in $\{x_1<0\}$
is implicit in the moving plane procedure.
\end{proof} 

We conclude this section with the proof of Theorem~\ref{Teo:simmetriamenoilpunto}.

\begin{proof}[Proof of Theorem \ref{Teo:simmetriamenoilpunto}]
   By Theorem~\ref{Teopmaggiore2} and Theorem~\ref{Teopminore2}, it follows that $u$ is symmetric with respect to 
the hyperplane $\{x_1=0\}$ and increasing in the $x_1$–direction in 
\(\mathbb{R}^N\cap\{x_1<0\}\). The symmetry result now follows in a standard way by 
performing the moving–plane procedure in any direction 
\(\nu\in\mathbb{S}^{N-1}\), and we finally deduce that $u$ is radial and radially decreasing. 
In particular, $u=u(r)$ with $u'(r)<0$ for every $r>0$; the strict inequality follows by ODE analysis. Now, the classification result is derived via the application of \cite[Theorem 5.2]{BV}.
\end{proof}

\textbf{Acknowledgements/Founding } The authors 
were partially supported by PRIN project
P2022YFAJH 003 (Italy): Linear and nonlinear PDEs; new directions and
applications. All the authors are partially supported also by Gruppo Nazionale
per l’Analisi Matematica, la Probabilit`a e le loro Applicazioni (GNAMPA) of
the Istituto Nazionale di Alta Matematica (INdAM). \\

\textbf{
Author contributions }All authors made substantial contributions to the conception or design of the work and approved the version to be published.\\

\textbf{Data Availability Statement}: All data generated or analyzed during this study
are included in this published article.\\

\end{document}